\long\def\skipit#1{} %can skip over \par
\def\={\,=\,}
\def\+{\,+\,}

\newcommand{\mdef}{\textit}

\newcommand{\ov}{\overline}

\newcommand{\uth}{^{\rm th}}

\documentclass[10pt]{amsart}
\usepackage{pstricks-add}
\usepackage{amsfonts}
\usepackage{amsmath}
\usepackage{amssymb}
\usepackage{ifthen,calc}
\usepackage{color}
\usepackage{epsfig}
\usepackage{graphicx}
\usepackage{epstopdf}
\usepackage{latexsym}
\usepackage{multicol}
\usepackage[sc]{mathpazo}
\usepackage[format=hang, margin=10pt]{caption}
\usepackage[mathscr]{eucal}

\newcounter{hours}
\newcounter{minutes}
\newcommand{\printtime}{
	\setcounter{hours}{\time/60}%
	\setcounter{minutes}{\time-\value{hours}*60}
	\ifthenelse{\value{hours}<10}{0}{}\thehours:%
	\ifthenelse{\value{minutes}<10}{0}{}\theminutes}

\numberwithin{equation}{section}
\numberwithin{figure}{section}
\numberwithin{table}{section}

\newtheorem{thm}{Theorem}[section]
\newtheorem{cor}[thm]{Corollary}
\newtheorem{lemma}[thm]{Lemma}
\newtheorem{prop}[thm]{Proposition}

\newtheorem{J-com}{JG-comment}[section]
\theoremstyle{definition}

\newtheorem{defn}{Definition}[section]

\def\dsum{\displaystyle\sum}

\keywords{Permutation-partition pair, Permutation-bipartition pair, Signed graph, embedding}

\begin{document}

\title {Permutation-bipartition pairs}%: A generalization for non-orientable embeddings }

%\skipit{
\author{Yichao Chen}
\address{School of Mathematics, SuZhou University of Science and Technology, Suzhou, 215009, China}
\email{chengraph@163.com}
%}  %% end-skipit

\thanks{This work is supported by the JSSCRC (Grant No. 2021530).}

\vskip.51cm
%\noindent \JGr{Version: \printtime\quad\today\quad}
\vskip.51cm

\begin{abstract}Permutation-partition pairs were introduced by Stahl in 1980. These pairs are generalizations of graphs  and graphs on surfaces. They were used to solve some problems for orientable embeddings of graphs. In this paper, we introduce a particular type of permutation-partition pair, called permutation-bipartition pair, which can be seen as generalizations of signed graphs and signed graph embeddings.  Some applications are given. %A walkup-reduction  for permutation-bipartition pairs  is presented.
\end{abstract}

\maketitle  %%%%%%%%%%%%

%%%%%%%%%%%%%%%%%%%%%%%%%%%%%%%%%%%%%%%%%%%%%%%%%%%%%%%%%%%%%%%%%%%%%%%%%%%%%%%%%%%%%%%%%%%%%%%%%%%%
%\bigskip
  %% Sec 1

\section{Basic Terminology and Background}

A \mdef{graph} $G=(V(G),E(G))$ is permitted to have  loops and multiple edges. We use $S$ to denote a\textit{ surface} without regard to orientability. An \textit{embedding} (or a map $M$) of  $G$  into a closed surface $S$ is a \textit{{cellular} embedding}.  A \textit{signed graph} $\Sigma=(G,\sigma)$ is a graph $G$ together with a mapping $\sigma$ which assigns $+1$ or $-1$ to each edge of the graph $G.$ If $\sigma(e)=1,$ we call the edge \textit{positive}, otherwise $\sigma(e)=-1$ and the edge is \textit{negative}. In this paper, the negative edges and positive edges in $\Sigma$ are represented by  dashed lines and solid lines respectively.

%% For any spanning tree of $G$, the number of co-tree edges is called the \mdef{{Betti number}} of $G$ and is denoted by $\beta(G)$.
 %In this paper, all graph embeddings are \textit{{cellular} embeddings}.
A \mdef{rotation at a vertex} $v$ of a graph $G$ is a cyclic ordering of the edge-ends incident at $v$.   A \mdef{ rotation system} $\rho$ of a graph $G$ is an assignment of a rotation at every vertex of $G$. It is well-known that there is an one-to-one
correspondence between the rotation systems and orientable embeddings.   A \textit{2-cell embedding} of a signed graph $\Sigma$ can be described combinatorially by a \textit{signed rotation system} $(\rho,\sigma),$ where $\rho$ is a rotation system of $G$ and $\sigma$ is the \textit{twist-indicator}, i.e., if $\sigma(e)=1$, then the edge $e$ is twisted; otherwise $\sigma(e)=0$ and $e$ is untwisted. It is obvious that if  $\sigma(e)=0$, for all $e\in E(G)$, then the signed rotation system $(\rho,\sigma)$ is {equivalent to} a  rotation system. Recall that any embedding of $G$ into a surface $S$ can be described by a signed rotation system. A signed rotation system is also known as a general rotation system in topological graph theory.
A \textit{switch} of a signed rotation system at a vertex $v$ means reversing the rotation about $v$ and changing the sign on each of its incident edges. Two signed rotation system are \textit{equivalent} if one can be transformed into another by a sequence of switches. For terms and definitions that are not explained here, we refer to the monograph on topological graph theory by Gross and Tucker \cite{GT87}.

Signed graphs were introduced by F. Harary in 1950s to model social relations involving disliking, indifference, and liking. There are many researches on signed graph embeddings, and we only list a few results here.  T. Zaslavsky considered orientation embedding of signed graphs, gave a characterization for the projective-planar signed graph (characterized by six small forbidden minors or eight small forbidden topological
subgraphs), and studied other topics for signed graph embeddings \cite{Zas92,Zas93,Zas96a,Zas97a,Zas97b}. \v{S}ir\'{a}n showed that Duke's classical theorem for graph embeddings does not extend to signed graph embeddings, and a signed graph orientation-embeds in only one surface if and only if two cycles are vertex disjoint \cite{Sir91a,Sir91b}. \v{S}ir\'{a}n and \v{S}koviera \cite{SS91} gave  characterizations of the maximum genus of a signed graph. Lv \cite{L15} calculated the largest demigenus for all signatures on $K_{3,n}.$ For more topics, we refer to the survey paper of Zaslavsky \cite{Zas98}.

A \textit{permutation-partition pair} $(P,\Pi)$ consists of an arbitrary permutation $P$ and an arbitrary partition $\Pi,$ both defined over some common underlying set $S.$ It can be seen as a combinatorial generalization of  graphs and graph embeddings. The notion of a permutation-partition pair was introduced by Stahl in 1980 \cite{Sta80}.  It is a useful tool to tackle orientable embeddings of a graph, such as minimum genus, maximum genus, average genus, and genus distribution, we refer the reader to \cite{Bon94,Sta82,Sta91a,Sta92,Sta97}, etc.

In this paper, we introduce the permutation-bipartition pair which is a generalization of signed graphs and signed graph embeddings. %(including non-orientable embeddings).
It can also be thought of as a generalization for permutation-partition pair to the non-orientable case. Our paper is organized as follows. In Section 2, we give a definition for permutation-bipartition pair, and introduce its embedding, region, and Euler-genus. We also prove a region (genus) Walkup reduction for embeddings of the permutation-bipartition pairs. It is then used to calculate region distributions (or Euler-genus polynomials) of linear signed graph families. The expected genus is discussed in Section 3.

\section{Permutation-bipartition pairs}

\subsection{Tutte's permutation axiomatization for embeddings}
Here we will introduce Tutte's permutation axiomatization for graph embeddings. It is our motivation and starting point  to introduce the permutation-bipartition pair. Suppose that $K=\{1,\alpha,\beta,\alpha\beta\}$ is the \textit{Klein four-group}. For an edge $a\in E,$ we introduce its \textit{two sides} and \textit{two ends}.  Assume $a$ itself at one end and on one side.  Let $\alpha$ be the permutation that interchanges symbols at the same end but different sides of an edge, for each edge. Let $\beta$ be the permutation that interchanges symbols at the same side but different ends of an edge, for each edge.  For $a\in E$, $Ka=\{a,\alpha a, \beta a,\alpha\beta a \}$ is called a \textit{quadricell}. Figure \ref{fig:tutte} shows an example for the above concepts.

 \begin{figure}[h]
\centering
\includegraphics[width=6cm]{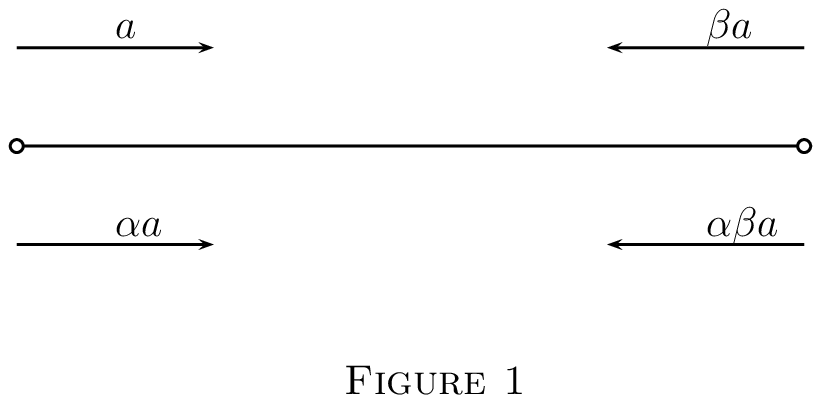}
\caption{}\label{fig:tutte}
\end{figure}

An embedding of $\Sigma$ on some surfaces induces a bi-rotation system $\rho$ on $V$ such that at each vertex, all its incident semi-edges are with a cyclic order, called a bi-rotation. Since each edge is considered as a quadricell, the bi-rotation at a vertex $v$ is $$\left\{(a,\rho a,\ldots,\rho^{(m-1)}a),(\alpha a,\alpha \rho^{-1} a,\ldots,\alpha \rho^{-(m-1)}a)\right\}.$$

For $k\geq1,$ let $\alpha_1,\alpha_2,\ldots,\alpha_k$ be permutations acting on the same set, we denote by $\langle \alpha_1,\alpha_2,\cdots,\alpha_k \rangle$ the group of permutations generated by them. Let $\iota$ be the \textit{identity permutation.} Tutte \cite{Tut84} considered graph embeddings as  permutations and provided an axiomatization for embeddings of $G$ in the following manner.
\begin{defn}\label{tutte}
A connected premap is an ordered triple $(\alpha, \beta,\rho)$ of permutations acting on a
set $S$ of $4m$ elements, such that
\begin{enumerate}
  \item  $\alpha^2=\beta^2=\iota, \alpha\beta=\beta\alpha=\gamma.$
  \item  for any $a\in S,$ then $a, \alpha a, \beta a, \alpha\beta a$ are distinct.
  \item  $\alpha \rho=\rho\alpha^{-1}$.
  \item  for each $a\in S$, the orbits of $\rho$ through $a$ and  $\alpha a$ are distinct.
  \item $\langle \alpha, \beta,\rho\rangle$ acts transitively on $S.$
\end{enumerate}
\end{defn}

\begin{thm} \label{tutte1}Let $M$ be a connected map in the sense of Definition \ref{tutte}, and let $S$ be a set of $4m$ symbols assigned bijectively to the side-end positions of $M.$ Let $\beta$ be the permutation that interchanges symbols at the same side but different ends of an edge, for each edge. Let $\alpha$ be the permutation that interchanges symbols at the same end but different sides of an edge, for each edge .
\begin{itemize}
  \item 	Vertices: Let $v$ be a vertex of $M$ and $(a_1,a_2,\ldots,a_{2k})$ the list of symbols encountered in a tour of the side-end positions incident with $v$ starting at an arbitrary symbol $a_1$, in the unique (local) direction such that $a_2 = \alpha a_1$. Then the permutation $\rho$ in Definition \ref{tutte} is the permutation whose disjoint cycles are associated in pairs with each vertex $v,$ and have the form $(a_1 a_3\cdots a_{2k-1})$ and $(a_{2k} a_{2k-2}\cdots a_2) = (\alpha a_{2k-1} \alpha a_{2k-3}\cdots\alpha a_1).$ The degree of $v$ is $k.$ (If $k=0$ we have a pair of empty cycles associated with the isolated vertex $v.$)
  \item Edges: For each $a\in S$, the elements of $a, \alpha a, \beta a, \alpha\beta a$ are the symbols assigned to the four side-end positions of the same edge.
  \item Faces: Let $f$ be a face of $M$ and $(b_1b_2\cdots b_{2j})$ the list of symbols encountered in a tour of the side-end positions incident with $f$ starting at an arbitrary symbol $b_1$, in the unique (local) direction such that $b_2 = \beta b_1$. Then the disjoint cycles of $\rho \alpha\beta$ are associated in pairs with each face $f,$ and have the form $(b_1 b_3\cdots b_{2j-1})$ and $(b_{2j} b_{2j-2}\cdots b_2) = (\beta b_{2j-1} \beta b_{2j-3} \cdots \beta b_1).$ The degree of $f$ is $j.$ (If $j = 0,$ we have a pair of empty cycles associated with the isolated face $f$.

\end{itemize}
\end{thm}

\begin{thm}\label{Tutte2}
Let $M=(\alpha, \beta, \rho, S )$ be a connected premap. If the action of $\langle \rho,\gamma\rangle$ on $S$ is transitive, then $M$ is non-orientable, otherwise there are exactly two orbits and $M$ is orientable.
\end{thm}
For more on Tutte's axiomatization of graph embeddings, see \cite{Liu09,Tut84}, etc.
\subsection{ Permutation-bipartition pairs}

 Suppose $S$ and $S_{\theta}$ are two disjoint finite sets such that $|S|=|S_{\theta}|.$ Let $\theta$ be a bijection from $S$ to $S_{\theta}.$ We call the ordered pair $(P, \Pi\cup\theta\Pi)$  a \textit{permutation-bipartition pair}, where the permutation $P$ is defined over $\textbf{S}=S\cup S_{\theta}$ such that if $(a,b,\ldots,f)$ is a cycle of $P$ then $(\theta a,\theta f,\ldots,\theta b)$ is also a cycle of $P$, and the partitions $\Pi=\{\Pi_i\}_{i=1}^{k}$ and $\theta\Pi=\{\theta\Pi_i\}_{i=1}^{k}$ are defined over $S$ and $S_{\theta}$ respectively. For each $i$ ($1\leq i\leq k$), we call the partition ${\Pi_i\cup \theta\Pi_i}$ a \textit{vertex }of the pair $(P, \Pi\cup\theta\Pi).$ For $a\in S$, we call the $\{a,\theta a\}$ a \textit{semi-edge} of the pair $(P, \Pi\cup\theta\Pi).$
 
 An \textit{embedding} of  a {permutation-bipartition pair}  $(P, \Pi\cup\theta\Pi)$ of $S\cup S_{\theta}$ is defined as a permutation $Q$  such that %each cyclic factor of the $Q$ agree with a member of the bipartition $\Pi\cup\theta\Pi$
 $Q$ is a bi-rotation defined on $\Pi\cup\theta\Pi,$ namely, if $\Pi_i=\{a_{i,1},a_{i,2},,\ldots,a_{i,k}\}$ and $\theta\Pi_i=\{\theta a_{i,1},\theta a_{i,2},\ldots,\theta a_{i,k}\} $ for $1\leq i \leq k,$ then \textit{bi-rotation} at $\Pi_i\cup \theta \Pi_i$ is $$\left\{(a_{i,j_1},a_{i,j_2},\ldots,a_{i,j_k}),(\theta a_{i,j_k},\ldots,\theta a_{i,j_2},a_{i,j_1})\right\},$$ where $j_1j_2\cdots j_k$ is a permutation on the set $\{1,2,\dots,k\}.$ We denote $S(\Pi\cup\theta\Pi)$ be the set of all bi-rotations defined on $\Pi\cup\theta\Pi.$

We now proceed to show that how an embedding $Q$ of $(P, \Pi\cup\theta\Pi)$ has a natural interpretation as embeddings of signed graphs. To each signed graph $\Sigma=(G,\sigma),$ we can associate a permutation-bipartition pair $(P_{\Sigma}, \Pi_{\Sigma}\cup\alpha\Pi_{\Sigma})$ in the following way. Let $M$ be the symmetric graph obtained by replacing each edge $a$ of $G$ with four elements $a, \alpha a, \beta a, \gamma a.$

Let

\begin{align*}
S\cup S_\alpha&=\bigcup_{a\in E(G)}Ka=\bigcup_{a\in E(G)}\{a, \alpha a, \beta a, \gamma a\},\\
P_\Sigma&=\gamma=\prod_{a\in E(G)}(a,\gamma a)( \alpha a, \beta a),\\
\Pi_{\Sigma}&=\bigcup_{v\in G}\Pi_v,\\
\alpha\Pi_{\Sigma}&=\bigcup_{v\in G}\alpha\Pi_v,
\end{align*}
where $\Pi_v=\{a,\rho a,\ldots,\rho^{(m-1)}a\},\alpha\Pi_v=\{\alpha a,\alpha \rho^{-1} a,\ldots,\alpha \rho^{-(m-1)}a\}.$

Fox example, suppose $\Sigma$ is the signed graph of Figure \ref{fig:mail} and all the signs of edges of $\Sigma$ are positive.  Then $$P_\Sigma=\gamma_{uv}\gamma_{vx}\gamma_{wu}\gamma_{wz}\gamma_{zu}\gamma_{zy}\gamma_{vy}\gamma_{xy},$$ and
$$\Pi_{\Sigma}=\{\Pi_u,\alpha\Pi_u,\Pi_v,\alpha\Pi_v,\Pi_x,\alpha\Pi_x,\Pi_y,\alpha\Pi_y,\Pi_w,\alpha\Pi_w,\Pi_z,\alpha\Pi_z\}$$
where  $\gamma_{uv}=(1,4)(2,3),\gamma_{vx}=(5,8)(6,7), \gamma_{wx}=(9,12)(10,11), \gamma_{wu}=(13,16)$ $(14,15),\gamma_{wz}=(17,20)(18,19),$ $\gamma_{zu}=(21,24)(22,23),\gamma_{zy}=(25,28)(26,27),$ $\gamma_{vy}=(29,32)(30,31),$ $\gamma_{xy}=(33,36)(34,35),$ $\Pi_u=\{2,15,23\},\alpha\Pi_u=\{1,16,24\},$ $\Pi_v=\{3,6,30\},\alpha\Pi_v=\{4,5,29\},$ $\Pi_x=\{7,9,35\},\alpha\Pi_x=\{8,10,36\},$ $\Pi_y=\{27,31,32\},\alpha\Pi_y=\{28,32,33\},$ $\Pi_w=\{12,14,18\},\alpha\Pi_w=\{11,13,17\},$ and $\Pi_z=\{19,22,26\},\alpha\Pi_z=\{20,21,25\}.$
 \begin{figure}[h]
\centering
\includegraphics[width=5cm]{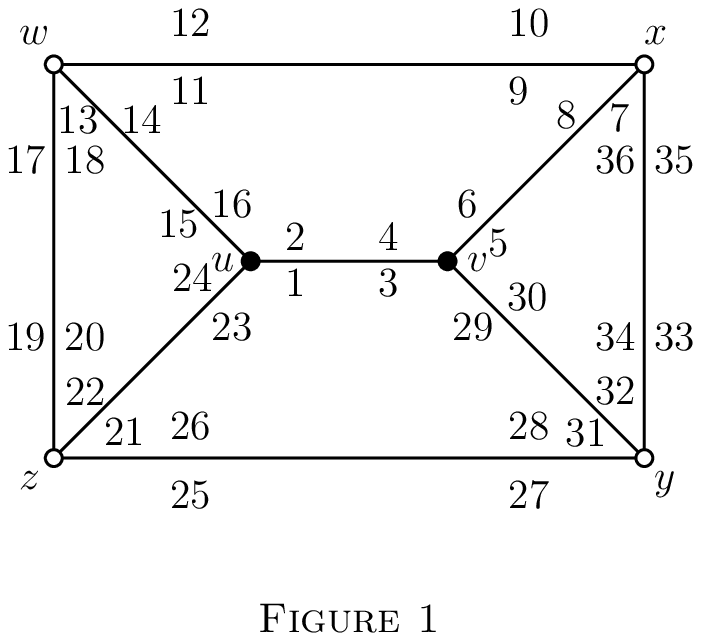}
\caption{}\label{fig:mail}
\end{figure}

A plane embedding of $\Sigma$ with

\begin{eqnarray*}P_{\Sigma}=&& (1,4)(2,3)(5,8)(6,7)(9,12)(10,11)(13,16)(14,15)(17,20)(18,19)\\&&(21,24)(22,23)(25,28)(26,27)(29,32)(30,31)(33,36)(34,35),
\end{eqnarray*}

\begin{eqnarray*}\label{Exam:1}Q_{\Sigma}=&&(1,16,24)(2,23,15)(5,4,29)(3,6,30)(10,8,36)(7,9,35)\\&&(11,17,13)(12,14,18)(21,20,25)(19,22,26)(33,32,28) (27,31,34)
\end{eqnarray*} and \begin{eqnarray*}P_{\Sigma}Q_{\Sigma}=&&(1,29,28,21)(3,23,26,31)(11,8,4,16)(9,14,2,6)\\
&&(12,35,27,19)(10,17,25,33)(5,36,32)(7,30,34)\\&&(13,24,20)(15,18,22).\end{eqnarray*}

Again, suppose $\Sigma$ is the signed graph of Figure \ref{fig:mail} and all edges of $\Sigma$ are positive except the edge $uv$. In this case, we have $$P_\Sigma=\gamma_{uv}\gamma_{vx}\gamma_{wu}\gamma_{wz}\gamma_{zu}\gamma_{zy}\gamma_{vy}\gamma_{xy},$$ where $\gamma_{uv}=(1,3)(2,4),\gamma_{vx}=(5,8)(6,7), \gamma_{wx}=(9,12)(10,11), \gamma_{wu}=(13,16)$ $(14,15),\gamma_{wz}=(17,20)(18,19),$ $\gamma_{zu}=(21,24)(22,23),\gamma_{zy}=(25,28)(26,27),$ $\gamma_{vy}=(29,32)(30,31),$ $\gamma_{xy}=(33,36)(34,35).$

%\begin{eqnarray*}P_{\Sigma}=&& (1,3)(2,4)(5,8)(6,7)(9,12)(10,11)(13,16)(14,15)(17,20)(18,19)\\&&(21,24)(22,23)(25,28)(26,27)(29,32)(30,31)(33,36)(34,35).\end{eqnarray*}

A projective plane embedding of $\Sigma$ with

\begin{eqnarray*}\label{Exam:1}Q_{\Sigma}=&&(1,16,24)(2,23,15)(5,4,29)(3,6,30)(10,8,36)(7,9,35)\\&&(11,17,13)(12,14,18)(21,20,25)(19,22,26)(33,32,28) (27,31,34)
\end{eqnarray*} and \begin{eqnarray*}P_{\Sigma}Q_{\Sigma}=&&(1,6,9,14,2,29,28,21)(3,16,11,8,4,23,26,31)\\
&&(12,35,27,19)(10,17,25,33)(5,36,32)(7,30,34)\\&&(13,24,20)(15,18,22).\end{eqnarray*}

 Suppose $b$ is a bit of $S,$ we denote by $P/b$ the permutation obtained by deleting $b$ in the disjoint cycle decomposition of $P.$ Similarly, $\Pi/b$ denotes the partition of $S-b$ by deleting the occurence of $b$ from the member of  $\Pi$ that contains it.   For $1\leq i\leq k,$ if $a,b\in S$ are in a member $\Pi_i$ of $\Pi$, we write $a\equiv b (mod\ \Pi).$  If $a\equiv b (mod\ \Pi),$ then $\theta a\equiv \theta b (mod\ \theta\Pi).$ A \textit{constraint} on the pair $(P, \Pi\cup\theta\Pi)$  is an ordered pair, denoted by $\{a\rightarrow b, \theta a\rightarrow \theta b\}$ such that $a\equiv b (mod\ \Pi).$

  A \textit{constraint set} $C$ on $(P, \Pi\cup\theta\Pi)$  is  a set given by $C=\{a_i\rightarrow b_i, \theta a_i\rightarrow \theta b_i\}_{i=1}^c,$ where the $a_i\rightarrow b_i$ are constraints. It is obvious that if the $a_i$ $(\theta a_i)$ are all distinct,  and the $b_i$ $(\theta b_i)$  are also all distinct.

Suppose that $C$ is a constraint set on  $(P, \Pi\cup\theta\Pi)$. We define  $$S(\Pi\cup\theta\Pi, C)=\{Q\in S(\Pi\cup\theta\Pi)| aQ=b, {\rm and}\ \theta b Q=\theta a\}$$

 If $a\equiv b (mod\ \Pi)$ then the pair $(P, \Pi\cup\theta\Pi)|_{\theta a\rightarrow \theta b}^{a\rightarrow b}$ is the pair $(P_{a,\theta a}, \Pi_{a}\cup\theta\Pi_{a}),$ where

\begin{align*}
\Pi_{a}&=\Pi-b, \theta\Pi_{a}=\theta\Pi-\theta b\\
P_{a}&=\left\{
                                                        \begin{array}{ll}
                                                         P(babP)/b& \hbox{if $a\neq bP\neq b$  ;} \\
                                                       P(ba)/b , & \hbox{if $a=bP\neq b$;}\\
                                                       P/b , & \hbox{if $bP= b$}.
\end{array}
                                                      \right.
\\
P_{a,\theta a}&=\left\{
                                                        \begin{array}{ll}
                                                   (\theta a,\theta b, \theta b P_a^{-1})P_a/\theta b & \hbox{if $\theta a\neq \theta b P_a^{-1}\neq \theta b$ ;} \\
                                                            (\theta b,\theta a )P_a/\theta b, & \hbox{if  $\theta a= \theta b P_a^{-1} \neq \theta b$;}\\
                                                            P_a/\theta b, & \hbox{if  $\theta b P_a^{-1} = \theta b.$}
                                                        \end{array}
                                                      \right.
\end{align*}

%A \textit{constraint}  on the pair $(P, \Pi\cup\theta\Pi)$ is an ordered pair such that $a\equiv b (mod\ \Pi),$ we denoted it by $a\rightarrow b.$  A constraint set C on (P, 11) is a set

Recall that the permutation $P_a$ was defined by Stahl. The definition for the permutation $P_{a,\theta a}$ from $P_a$ here is new. It's a crucial step to express the embeddings of a permutation-bipartition pair in terms of those of smaller pairs.

 The two permutations  $P_a$  and $P_{a,\theta a}$ above actually have two  natural visual interpretations. If  $a$ and $b$ belong to the same orbit $B = (a c\cdots d b e \cdots f)$ of $P,$ then $P_a$ is obtained from $P$ by splitting $B$ (at $a$ and at $b$) into two cycles, and deleting $b$ in order to obtain two orbits $B_1B_2= (a c \cdots d)(e \cdots f),$ and all the other orbits of $P$ inherit completely as $P_a$.
 If  $a$ and $b$ belong to distinct orbits $B_1= (a c \cdots d)$ and $B_2 = (b e\cdots f)$ of $P,$ then $P$ is obtained by combining $B_1$, and $B_2$ into a single cycle and deleting $b$ in order to get the cycle $B=(a c \cdots d e \cdots f),$ and all the other orbits of $P$ inherit completely by $P_a.$

Again, here's an intuitive explanation for deriving $P_{a,\theta a}$ from $P_a.$ If  $\theta a$ and $\theta b$ belong to the same orbit $B = (\theta a, c^{'},\cdots, d^{'},\theta b, e^{'}, \cdots, f^{'})$ of $P_a,$ then we swap $\theta a$ and $\theta b$ in $B$ so as to get $B^{'}=(\theta b, c^{'},\cdots, d^{'},\theta a, e^{'} \cdots f^{'})$. And $P_{a,\theta a}$ is obtained from $P_a$ by splitting $B{'}$ (at $\theta a$ and at $\theta b$) into two cycles, and deleting $\theta b$ in order to obtain two orbits $B_1B_2= (c^{'} \cdots d^{'})(\theta a, e^{'}, \cdots, f^{'}),$ and all the other orbits of $P_{a,\theta a}$ inherit completely as $P_a$. If  $\theta a$ and $\theta b$ belong to distinct orbits $B_1= (\theta a, c^{'},\cdots, d^{'})$ and $B_2 = (\theta b, e^{'}, \cdots, f^{'})$ of $P_a,$ again, we swap $\theta a$ and $\theta b$ in $B_1$ and $B_2$ so as to get $B_1^{'}=(\theta b, c^{'},\cdots, d^{'})$ and $B_2^{'}=(\theta a, e^{'} \cdots f^{'})$. $P_{a,\theta a}$ is obtained by combining $B_1^{'}$, and $B_2^{'}$ into a single cycle and deleting $\theta b$ in order to get the cycle $B=(\theta a, e^{'} \cdots f^{'},c^{'},\cdots, d^{'})$ and all the other orbits of $P_a$ inherit completely by $P_{a,\theta a}.$

\begin{prop} If $(P, \Pi\cup\theta\Pi)$ is a permutation-bipartition pair, then $(P_{a,\theta a}, \Pi_{a}\cup \theta \Pi_{a})$ is also a permutation-bipartition pair.
\end{prop}
\begin{proof}The proof have three cases. If $a$ and $b$ belong to the same orbit $B = (a c\cdots d b e \cdots f)$ of $P,$ then we can suppose that $$P=B(\theta B^{-1})\overline{P}=(a c\cdots d b e \cdots f)(\theta a, \theta{f},\cdots, \theta{e},\theta b, \theta{d},\cdots, \theta{c})\overline{P}$$
By the discussion above, we have $P_{a}=(a c \cdots d)(e \cdots f)(\theta a, \theta{f},\cdots, \theta{e},\theta b, \theta{d},\cdots, \theta{c})\overline{P},$ and $P_{a,\theta a}=(a c \cdots d)(e \cdots f)(\theta{f},\cdots, \theta{e})(\theta a, \theta{d},\cdots, \theta{c})\overline{P}.$ 

If $a$, $b$, $\theta a$ and $\theta b$ belong to four distinct orbits of $P$, then we can suppse that $a$ and $b$ belong to distinct orbits $B_1=(a c \cdots d)$ and $B_2 = (b e\cdots f)$ of $P,$ we let $$P=B_1B_2 (\theta B_1^{-1})(\theta B_2^{-1})\overline{P}=(a c \cdots d)(b e\cdots f)(\theta a, \theta{d},\cdots, \theta{c})(\theta b,\theta{f},\cdots, \theta{e})\overline{P},$$
and we have  $P_{a,\theta a}=(a c \cdots d e \cdots f)(\theta a, \theta{f},\cdots, \theta{e},\theta{d},\cdots, \theta{c})\overline{P}.$

If $a$, $b$, $\theta a$ and $\theta b$ belong to  two distinct orbits of $P,$ by symmetry, we can suppose that $a$ and $\theta b$ belong to orbit $B=(a,c,d,\cdots,\theta b,e,\cdots,f)$ of $P$, then we let $$P= B(\theta B^{-1})\overline{P}=(a, c, \cdots d,\theta b, e,\cdots, f)(\theta a, \theta{f},\cdots, \theta{e},b,\theta{d},\cdots, \theta{c})\overline{P},$$ and again we have $P_{a,\theta a}=(a, c, \cdots d,\theta{f},\cdots, \theta{e})(\theta a,e,\cdots, f,\theta{d},\cdots, \theta{c})\overline{P}.$

\end{proof}
 
 For example, if $(P, \Pi\cup\theta\Pi)= ((1 4  5 7)(2 8 6 3),\ \ \{\{1,3,5,7\},\{2,4,6,8\}\}),$ where $\theta: i\rightarrow i+1,$ for $i=1,3,5,7.$ Then

 \begin{eqnarray*} (P, \Pi\cup\theta\Pi)|_{2\rightarrow 4}^{1\rightarrow 3}&=& \left((1 8 6) (2 5 7),\ \ \{\{1,5,7\},\{2,6,8\}\}\right)\\
      (P, \Pi\cup\theta\Pi)|_{2\rightarrow 6}^{1\rightarrow 5}&=& ((1 4)(7) (2 3)(8),\ \ \{\{1,3,7\},\{2,4,8\}\})\\
      (P, \Pi\cup\theta\Pi)|_{2\rightarrow 8}^{1\rightarrow 7}&=& ((1 4 5 )(2 6 3),\ \ \{\{1,3,5\},\{2,4,6\}\})\\
      (P, \Pi\cup\theta\Pi)|_{4\rightarrow 6}^{3\rightarrow 5}&=& ((3 2 8)(4 7 1),\ \ \{\{1,3,7\},\{2,4,8\}\})\\
      (P, \Pi\cup\theta\Pi)|_{4\rightarrow 8}^{3\rightarrow 7}&=& ((5 3 2) (4 6 1),\ \ \{\{1,3,5\},\{2,4,6\}\})\\
      (P, \Pi\cup\theta\Pi)|_{8\rightarrow 2}^{7\rightarrow 1}&=& ((7) (4 5) (8)(6 3),\ \ \{\{3,5,7\},\{4,6,8\}\})\\
      (P, \Pi\cup\theta\Pi)|_{8\rightarrow 6}^{7\rightarrow 5}&=& ((7 1 4) (8 3 2),\ \ \{\{1,3,7\},\{2,4,8\}\})\\
                             \end{eqnarray*}
%$(P, \Pi\cup\theta\Pi)|_{2\rightarrow 4}^{1\rightarrow 3}=((1 2) (5 6 7)(8),\{\{1,5,7\},\{2,6,8\}\})$

%\section{Some lemmas}

Suppose $P$ is a permutation on $S$, we denote $\|P\|$ as the number of cycles in the disjoint cycle decomposition of $P.$ The following two lemmas are analogs of Lemma 1.1 and Lemma 1.2 in \cite{Sta91a}. They play a similar role in calculating region distributions of permutation-bipartition  pairs. They also can be used to tackle some problems for non-orientable embeddings of a graph.

\begin{lemma}\label{lem:S1}   Suppose that $(P, \Pi\cup\theta\Pi)$ is a pair with bits $a\neq b, \theta a\neq \theta b$ and let $(P_{a,\theta a}, \Pi_{a}\cup \theta \Pi_{a}), = (P, \Pi\cup\theta\Pi)|_{\theta a\rightarrow \theta b}^{a\rightarrow b}$. Then the function
	$$f: S(\Pi\cup \theta \Pi, a\rightarrow b,  \theta a\rightarrow \theta b)\rightarrow S(\Pi_{a}\cup \theta \Pi_{a})$$
given by

$$f(Q)=\left(\left[(a b)Q/b\right](\theta a,\theta b)/\theta b\right)=Q_{a,\theta a}$$

is a bijection such that
$$\|PQ\|= \|P_{a,\theta a}Q_{a,\theta a}\|+\delta_{a,bP}+\delta_{\theta aP_a,\theta b},$$ where $\delta$ is the Knonecker delta function.
\end{lemma}

\begin{proof}We first show that the function $f$ is a bijection. Recall that for any $Q\in S(\Pi\cup \theta \Pi)$, $(a b)Q(\theta a,\theta b)$ has $(b)$ and $(\theta b)$ as singletons, and hence $f$ is  injective.  On the other hand,  for any $Q_{a,\theta a}\in S(\Pi_{a}\cup \theta \Pi_{a})$, $(a b)Q_{a,\theta a}(\theta a,\theta b)$ maps $a$ to $b$ and $\theta b$ to $\theta a,$ and we have $$f((a b)Q_{a,\theta a}(\theta a,\theta b))=Q_{a,\theta a}.$$  Thus $f$ is surjective.

Recall that \begin{align*}
P_{a}&=\left\{
                                                        \begin{array}{ll}
                                                         P(babP)/b& \hbox{if $a\neq bP\neq b$  ;} \\
                                                       P(ba)/b , & \hbox{if $a=bP\neq b$;}\\
                                                       P/b , & \hbox{if $bP= b$}.
\end{array}
                                                      \right.
\end{align*}
and
\begin{align*}
P_{a,\theta a}&=\left\{
                                                        \begin{array}{ll}
                                                   (\theta a,\theta b, \theta b P_a^{-1})P_a/\theta b & \hbox{if $\theta a\neq \theta b P_a^{-1}\neq \theta b$ ;} \\
                                                            (\theta b,\theta a )P_a/\theta b, & \hbox{if  $\theta a= \theta b P_a^{-1} \neq \theta b$;}\\
                                                            P_a/\theta b, & \hbox{if  $\theta b P_a^{-1} = \theta b.$}
                                                        \end{array}
                                                      \right.
\end{align*}
%We have the following nine cases.

   By Lemma 1.1 in \cite{Sta91a}, we have $$\|PQ\|=\|P_{a}Q_{a}\|+\delta_{a,bP}.$$ If $\theta a\neq \theta bP_a^{-1}\neq \theta b,$ note that $b$ ($\theta b$) is in the domain of neither $P_a$($P_{a,\theta a}$) nor $Q_a$ ($Q_{a,\theta a}$), then
   \begin{eqnarray*} \|PQ\|&=& \|P_{a}Q_{a}\|+\delta_{a,bP}\\
      &=& \|(\theta bP_{a}^{-1},\theta b, \theta a)P_{a,\theta a}Q_{a,\theta a}(\theta a,\theta b)\|+\delta_{a,bP}\\
      &=&\|(\theta b, \theta aP_{a}^{-1})P_{a,\theta a}Q_{a,\theta a}\|+\delta_{a,bP}\\
      &=&\|P_{a,\theta a}Q_{a,\theta a}\|+\delta_{a,bP}.
      \end{eqnarray*}
  Furthermore, if $\theta a= \theta bP_a^{-1}\neq \theta b,$ then \begin{eqnarray*} \|PQ\|&=&\|P_{a}Q_{a}\|+\delta_{a,bP}\\
      &=& \|(\theta b, \theta a)P_{a,\theta a}Q_{a,\theta a}(\theta a,\theta b)\|+\delta_{a,bP}\\
      &=&\|(\theta b)P_{a,\theta a}Q_{a,\theta a}\|+\delta_{a,bP}\\
      &=&\|P_{a,\theta a}Q_{a,\theta a}\|+\delta_{a,bP}+1.
            \end{eqnarray*}
  If  $\theta b P_a^{-1} = \theta b,$ we have \begin{eqnarray*} \|PQ\|&=&\|P_{a}Q_{a}\|+\delta_{a,bP}\\
      &=& \|(\theta b)P_{a,\theta a}Q_{a,\theta a}(\theta a,\theta b)\|+\delta_{a,bP}\\
      &=&\|(\theta a,\theta b)P_{a,\theta a}Q_{a,\theta a}\|+\delta_{a,bP}\\
      &=&\|P_{a,\theta a}Q_{a,\theta a}\|+\delta_{a,bP}.
      \end{eqnarray*}

We have the desired result.
\end{proof}

\begin{lemma}\label{lem:S2}    Suppose that $(P, \Pi\cup\theta\Pi)$ is a pair such that $\{b\}\in \Pi,$ $\{\theta b\}\in \theta \Pi,$  and let
 $(P_{b,\theta b}, \Pi_{b}\cup \theta \Pi_{b}), = (P, \Pi\cup\theta\Pi)|_{\theta b\rightarrow \theta b}^{b\rightarrow b}.$ Then the function

	$$f: S(\Pi\cup \theta \Pi, b\rightarrow b,  \theta b\rightarrow \theta b)\rightarrow S(\Pi_{b}\cup \theta \Pi_{b})$$
given by

$$f(Q)=\left[Q/b\right]/\theta b=Q_{b,\theta b}$$
is a bijection such that
$$\|PQ\|= \|P_{b,\theta b}Q_{b,\theta b}\|+\delta_{b,bP}+\delta_{\theta bP_b,\theta b}.
$$
\end{lemma}

\begin{proof}It is quite obvious that the function $f$ is  a bijection, since $Q_{b,\theta b}$ is obtained from $Q$ by deleteing the orbits $(b),$ and $(\theta b).$  By Lemma 1.2 in \cite{Sta91a}, $$\|PQ\|=\|P_{b}Q_{b}\|+\delta_{b,bP}.$$

If $\theta b \neq \theta bP_b^{-1}$, then $P_{b,\theta b}=(\theta b,\theta b P_b^{-1})P_b/\theta b = P_{b,\theta b}/\theta b,$ where $(\theta b)$ is an orbit of $P_{b,\theta b}$. Since $\theta b$ is in the domain of neither $P_{b,\theta b}$ nor $Q_{b,\theta b}$, in this case

\begin{eqnarray*} \|PQ\|&=& \|P_{b}Q_{b}\|+\delta_{b,bP}\\
      &=& \|(\theta b,\theta b P_b^{-1})P_{b,\theta b}Q_{b,\theta b}(\theta b)\|+\delta_{b,bP}\\
      &=&\|P_{b,\theta b}Q_{b,\theta b}\|+\delta_{b,bP}.
          \end{eqnarray*}
Otherwise  $\theta b \neq \theta bP_b^{-1}$, in this case $(\theta b)$ is an orbit of both $P_b$ and $Q_b$, and it's obvious that
\begin{eqnarray*} \|PQ\|&=& \|P_{b}Q_{b}\|+\delta_{b,bP}\\
      &=& \|(\theta b)P_{b,\theta b}Q_{b,\theta b}(\theta b)\|+\delta_{b,bP}\\
      &=&\|P_{b,\theta b}Q_{b,\theta b}\|+\delta_{b,bP}+1.
      \end{eqnarray*}
\end{proof}

An Euler digraph is a digraph where every vertex has in-degree equal to its out-degree. Let $Q$ be a bi-rotation of the pair $(P, \Pi\cup\theta\Pi).$  We associate with $(P,Q)$ an embedding of 4-regular Euler digraph $D(P, Q)$ whose vertices are the elements of $\textbf{S}.$ The arbitrary
vertex $u$ of $D(P, Q)$ has two arcs emanating from it, one going to $vP$ and the other to $uQ.$ Furthermore, the rotation of $v$ is $(vP^{-1},vP ,vQ^{-1}, vQ).$ It has $|\textbf{S}|$ vertices, $2|\textbf{S}|$ arcs, and $||P||+||Q||+||PQ||$ regions. From Euler's formula, it's Euler characteristic equals $$|\textbf{S}|-2|\textbf{S}|+||P||+||Q||+||PQ||$$

Since $|\textbf{S}|$, $||P||$, $||Q||$ and Euler characteristic are all even numbers, it follows that $||PQ||$ is also an even number.

For example, Suppose the pair  $(P, \Pi\cup\theta\Pi)=((1,\theta 2,3)(\theta 1,\theta 3,2 ), \{\{1,2,3\}\cup \{\theta 1,\theta 2, \theta 3\}).$ Let $P=(1,\theta 2,3)(\theta 1,\theta 3,2 ),\ Q=(1,2,3)(\theta 1,\theta 3,\theta 2),$ then an embedding of $D(P,Q)$ is shown in Figure \ref{digraph}. It has $6$ vertices, $12$ arcs, and $6$ regions. The number of connected components of $D(P, Q)$ is $1.$
 \begin{figure}[h]
\centering
\includegraphics[width=8cm]{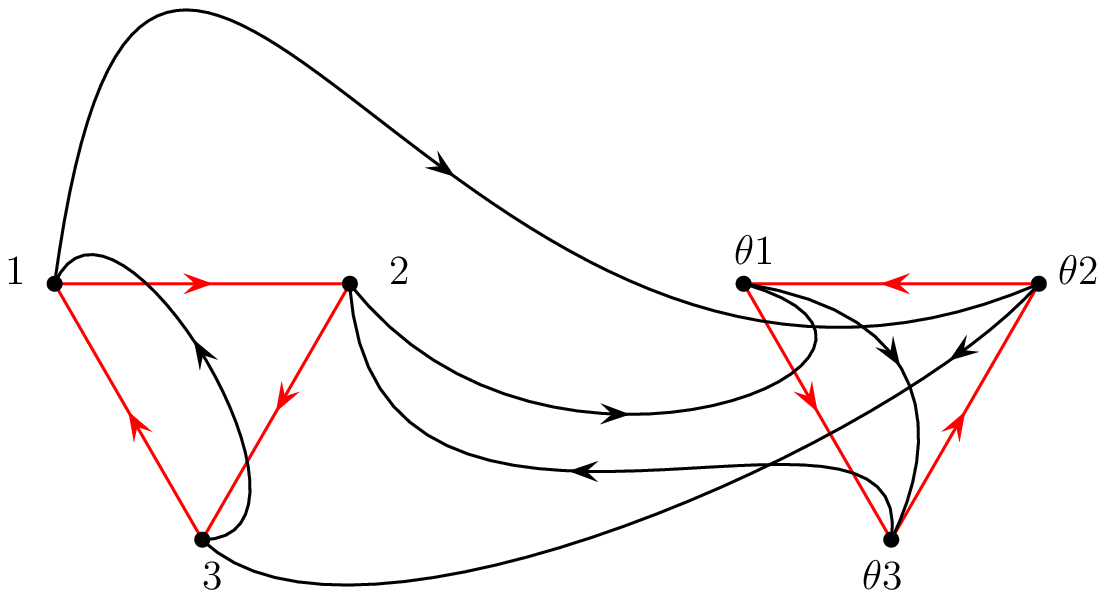}
\caption{}\label{digraph}
\end{figure}

If the pair $(P, \Pi\cup\theta\Pi)$ is $((1,2,3)(\theta 1,\theta 3,\theta 2 ), \{\{1,2,3\}\cup \{\theta 1,\theta 2, \theta 3\}).$ Let $P=Q=(1,2,3)(\theta 1,\theta 3,\theta 2),$ then the number of connected components of $D(P, Q)$ is $2$ as shown in Figure \ref{digraph:BB3}. The corresponding embedding of $D(P,Q)$ has $6$ vertices, $12$ arcs, and $6$ regions. 

\begin{figure}[h]
\centering
\includegraphics[width=4cm]{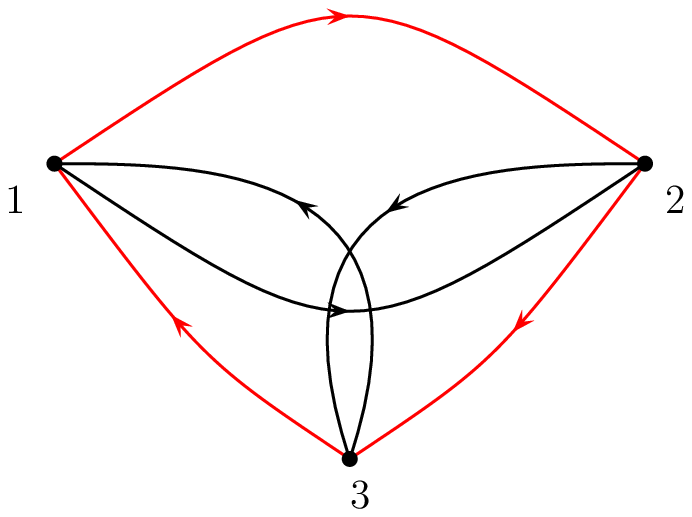}
\includegraphics[width=4.2cm]{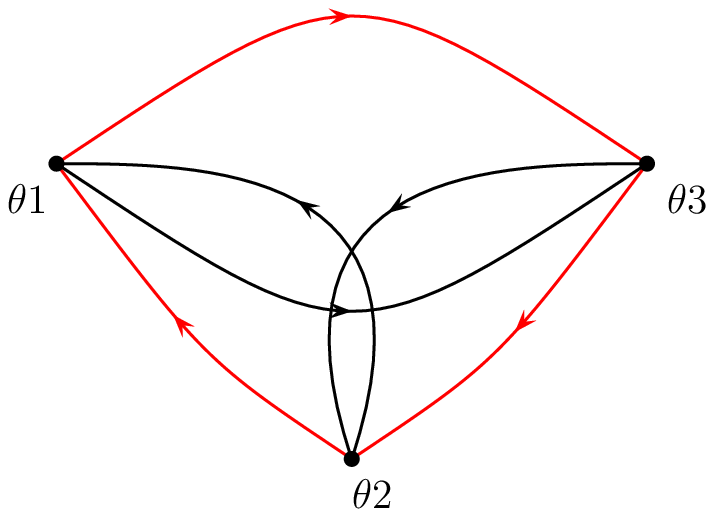}
\caption{}\label{digraph:BB3}
\end{figure}

The \textit{orbit distribution} of the pair $(P, \Pi\cup\theta\Pi)$ is define as $r_{(P, \Pi\cup\theta\Pi)}(2k)=\{Q\in S(\Pi\cup\theta\Pi):\ \|PQ\|=2k\}.$
\begin{cor} \label{cor:1}If  $(P, \Pi\cup\theta\Pi)$ is a permutation-bipartition pair such that $b\in S,$ $\{b\}\notin \Pi $ and $\theta b\in S_{\theta},$ $\{\theta b\}\notin \theta\Pi $
 then
	$$r_{(P, \Pi\cup\theta\Pi)}(2k)= \dsum_{\begin{array}{c}
a\equiv b (mod\ \Pi_i), \ a\neq b\\
\theta a\equiv \theta b (mod\ \theta\Pi_i), \ \theta a\neq \theta b
\end{array}} r_{(P, \Pi\cup\theta\Pi)|_{\theta a\rightarrow \theta b}^{a\rightarrow b}}(2k-\delta_{a,bP}-\delta_{\theta aP_a,\theta b}).$$

\end{cor}

\begin{cor} \label{cor:2}If  $(P, \Pi\cup\theta\Pi)$ is a permutation-bipartition pair such that $\{b\}\in \Pi, \{\theta b\}\in \theta\Pi ,$
 then
	$$r_{(P, \Pi\cup\theta\Pi)}(2k)=  r_{(P, \Pi\cup\theta\Pi)|_{\theta b\rightarrow \theta b}^{b\rightarrow b}}(2k-\delta_{b,bP}-\delta_{\theta b,\theta bP_b}).$$
\end{cor}

 %We can see that that each vertex of $D(Q, P)$ has both indegree and outdegree $2.$
%It is easy to see that $D(P, Q) = D(Q, P),$ and that $c(P, Q)$ is the number of connected components of $D(P, Q).$ 

Note that such directed embeddings of Euler digraphs (Tutte called them plane alternating dimaps) were studied  by Tutte \cite{Tut48} in 1948.  Bonnington et al. \cite{BCMM02} made a  systematic study of {directed embeddings} of an Eulerian digraph into surfaces.

\subsection{A region Walkup reduction}\label{region}
The two corollaries above enable us to expresses the regions (or faces) of the embeddings of a pair $(P, \Pi\cup\theta\Pi)$ in terms of those of smaller pairs. It therefore makes possible inductive proofs for signed graph embeddings. We can associate with each pair $(P, \Pi\cup\theta\Pi)$  a \textit{region reduction diagram} $\mathfrak{F}_{(P, \Pi\cup\theta\Pi)}$ with the edge label $\delta_{a,bP}+\delta_{\theta aP_a,\theta b}$ which is similar to the reduction diagram of \cite{Sta91a} (Called Walkup reduction in \cite{Sta91a} by Stahl).

Let $s= (b_1\prec  b_2  \prec\cdots\prec b_n)$ be any linear ordering of the set $\Pi$ that underlies the pair $(P, \Pi\cup\theta\Pi)$ and let $\theta s=(\theta b_1\prec \theta b_2 \prec\cdots\prec \theta b_n)$ be the linear ordering of the set $\theta\Pi.$ Let $\mathfrak{F}_{(P, \Pi\cup\theta\Pi)}^{0}$ be $(P, \Pi\cup\theta\Pi),$ and assuming that the vertex set $\mathfrak{F}_{(P, \Pi\cup\theta\Pi)}^{i}$ was given, for $0\leq i< n.$ Let $(P^{'},  \cup\theta\Pi^{'})$ be any vertex in $\mathfrak{F}_{(P, \Pi\cup\theta\Pi)}^{i}.$ If $\{b_{i+1}\}$ and $\{\theta b_{i+1}\}$ are  singleton members of $\Pi^{'}$ and $\theta\Pi^{'},$ respectively. Then $(P^{'},  \Pi^{'}\cup\theta\Pi^{'})$ has only one descendent $(P^{'}, \Pi^{'} \cup\theta\Pi^{'})|_{\theta \theta b\rightarrow \theta b}^{b\rightarrow b};$ Otherwise $\{b_{i+1}\}$ and $\{\theta b_{i+1}\}$ are not singleton members of $\Pi^{'}$ and $\theta\Pi^{'},$ respectively, then each of the pairs $(P^{'},  \Pi^{'}\cup\theta\Pi^{'})|_{\theta a\rightarrow \theta b_{i+1}}^{a\rightarrow b_{i+1}}$ is a descendent of $(P^{'},  \Pi^{'}\cup\theta\Pi^{'})$, for  $a\equiv b (mod\ \Pi_i),$ and $\theta a\equiv \theta b (mod\ \theta\Pi_i)$. Each branch from  $(P^{'},  \Pi^{'}\cup\theta\Pi^{'})$ to any of its descendent  $(P^{'},  \Pi^{'}\cup\theta\Pi^{'})|_{\theta a\rightarrow \theta b_{i+1}}^{a\rightarrow b_{i+1}}$ is assigned the weights $\delta_{a,bP}+\delta_{\theta aP_a,\theta b}$.

The following theorem follows from Corollary \ref{cor:1}  and Corollary \ref{cor:2}.

\begin{thm}\label{region}
The embeddings of the permutation-bipartition pair $(P, \Pi\cup\theta\Pi)$ are in a one-to-one corresponding with the directed path of $\mathfrak{F}_{(P, \Pi\cup\theta\Pi)}.$ This correspondence is such that the sum of the weights along the corresponding path is twice the number of regions of the embedding $(P, \Pi\cup\theta\Pi).$
\end{thm}

\begin{figure}[h]
\centering
\includegraphics[width=8cm]{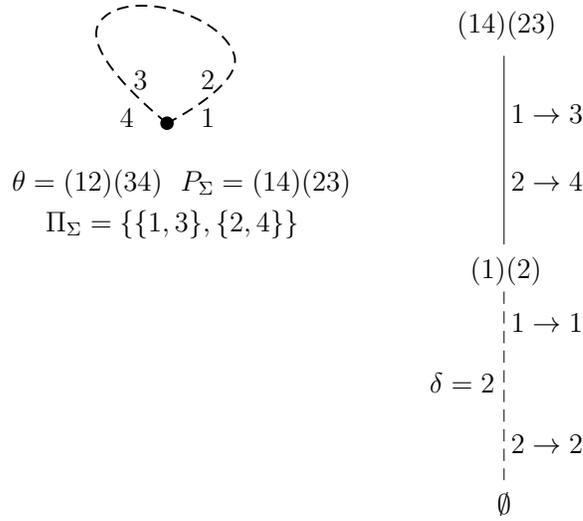}
\caption{A twisted loop and its reduction diagram.}\label{fig:reB1}
\end{figure}

Figure \ref{fig:reB1} shows a signed graph $\Sigma$ of one vertex with a negative edge (namely, $\Sigma$ is a twisted loop). By the reduction diagram, we know  $\Sigma$ embeds on projective plane with one region.
 \begin{figure}[h]
\centering
\includegraphics[width=8cm]{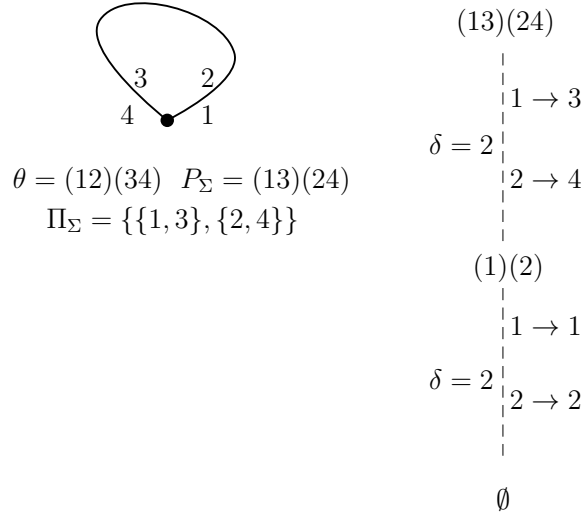}
\caption{An oriented loop and its reduction diagram.}\label{fig:reB1-o}
\end{figure}

 If we change the sign of $\Sigma$  in Figure \ref{fig:reB1},  then $\Sigma$ will be change to a single loop. The reduction diagram of Figure \ref{fig:reB1-o} shows that it embeds on the plane with two regions.

\begin{figure}[h]
\centering
\includegraphics[width=4.5cm]{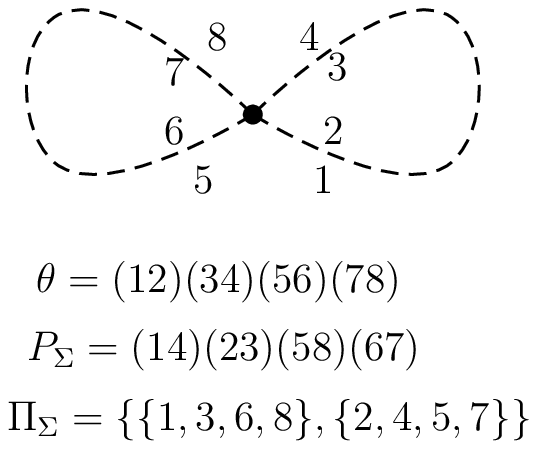}
\includegraphics[width=11cm]{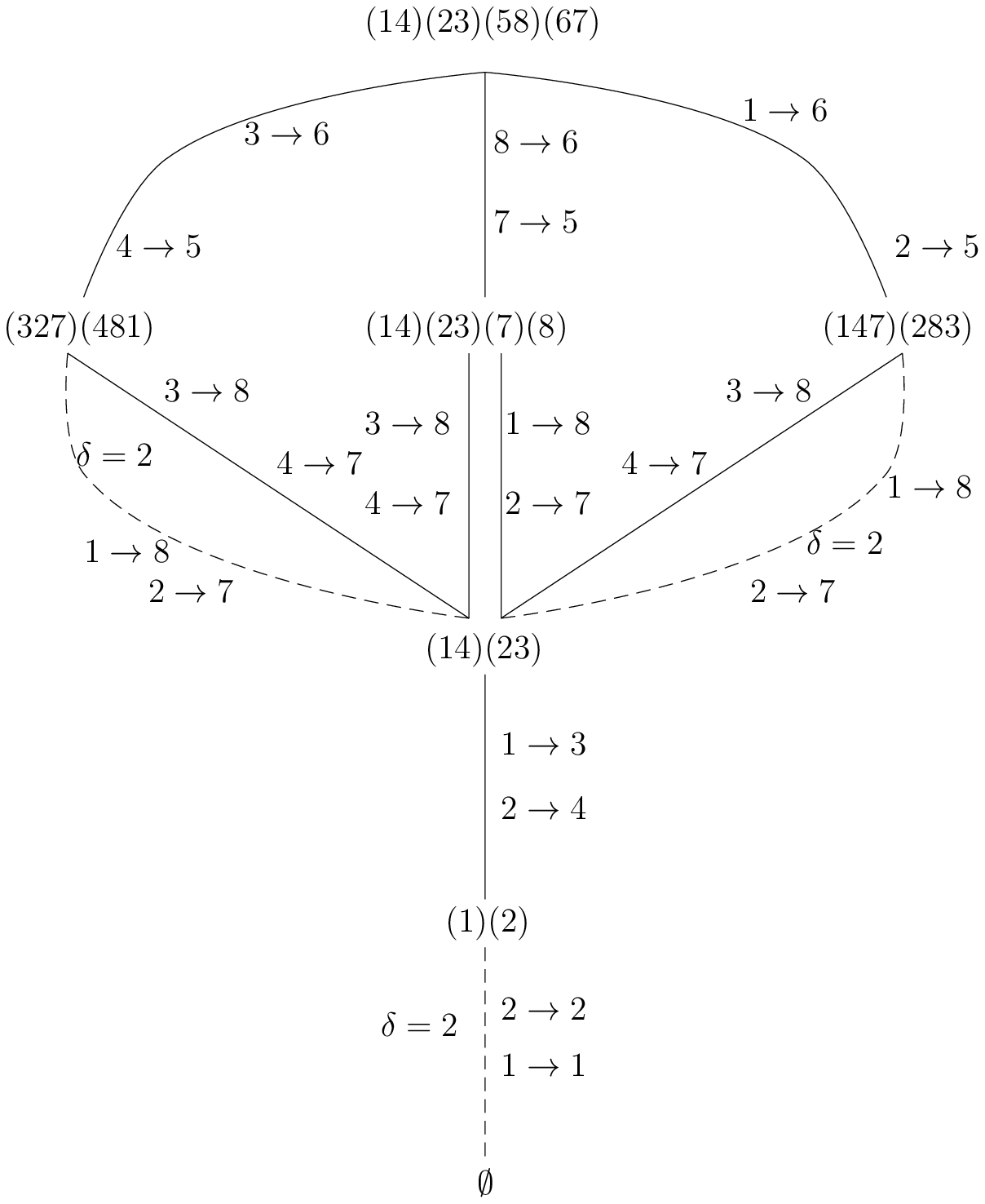}
\caption{A signed graph $\Sigma$ with one vertex and two twisted loops, and its reduction diagram.}\label{fig:reB2C}
\end{figure}

 Again, let $\Sigma$ be the signed graph of Figure \ref{fig:reB2}, it contains one vertex with two negative edges. The reduction diagram shows that it has two embeddings on projective plane with two regions, and four embeddings on Klein bottle with one region. The signed graph in Figure \ref{fig:reB2C} differs from the signed graph in Figure \ref{fig:reB2} in that it contains a negative edge and a positive edge, it has four embeddings on projective plane with two regions, and two embeddings on Klein bottle with one region.

\begin{figure}[h]
\centering
\includegraphics[width=4.5cm]{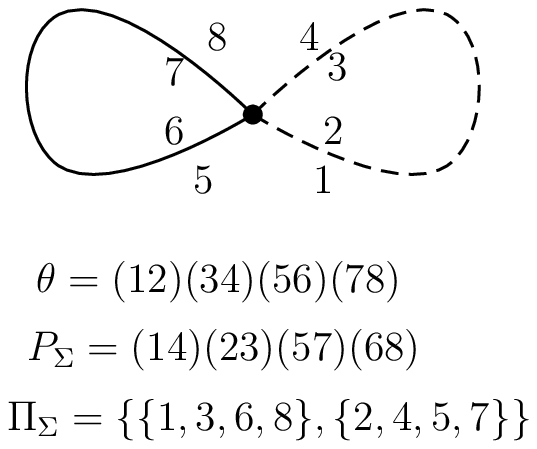}
\includegraphics[width=11cm]{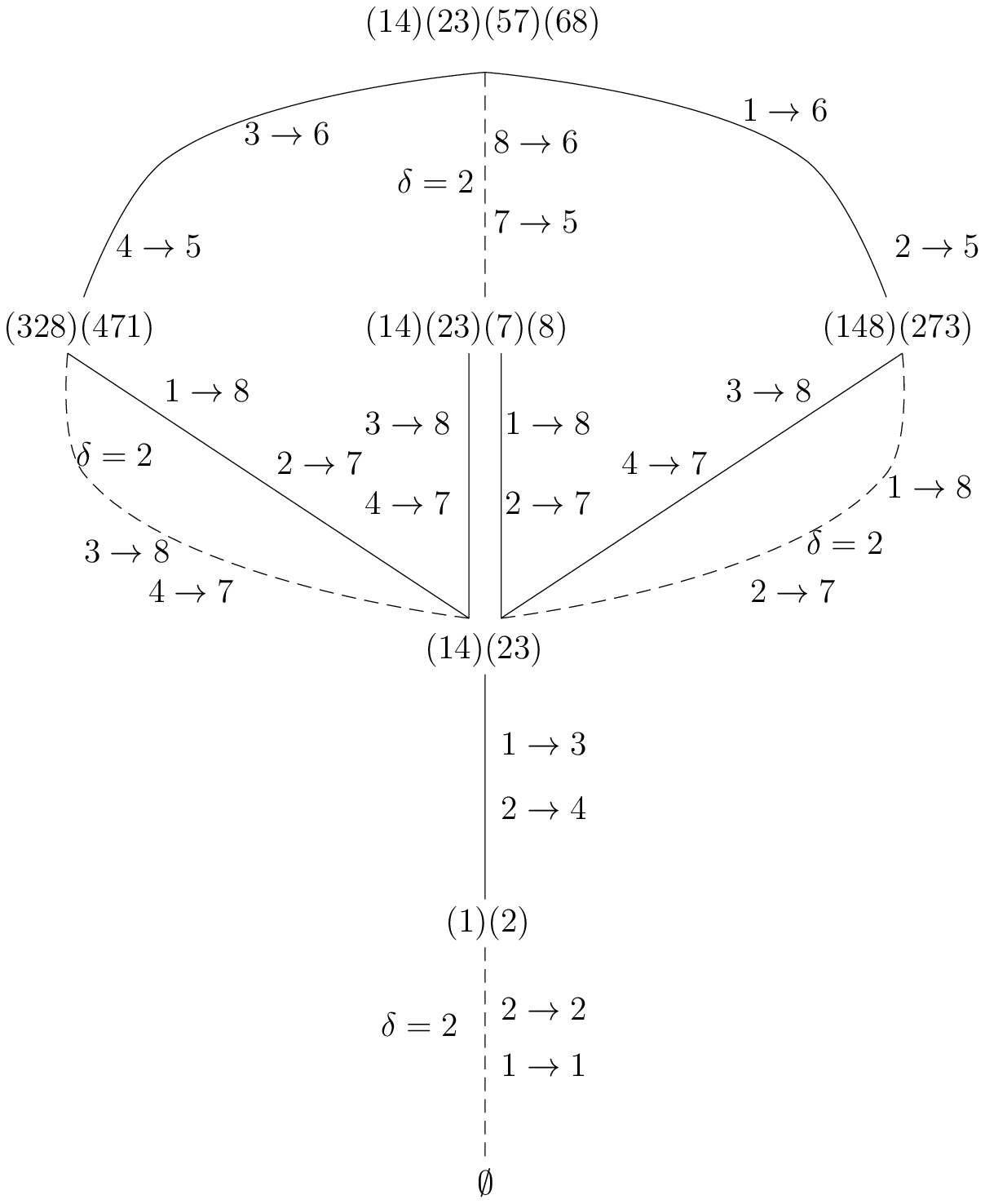}
\caption{A signed graph $\Sigma$ with one vertex, a loop and a twisted loop, and its reduction diagram.}\label{fig:reB2}
\end{figure}

%\section{An upper bound for the expected regions of a signed graph}

\subsection{The genus version of Walkup reduction}\label{region} We call a permutation-bipartition pair $(P, \Pi\cup\theta\Pi)$  \textit{non-orientable}, if there exits a permutation $Q_i$ on $\Pi_i,$ and a permutation $Q_j$ on $\theta \Pi_j,$ such that the action of $\langle P,Q_i,Q_j\rangle$ on $\Pi_i\cup\theta \Pi_j$ is transitive for $1\leq i,j\leq k$. Otherwise, the pair $(P, \Pi\cup\theta\Pi)$ is an \textit{orientable permutation-bipartition pair}. The\textit{ Euler characteristic} of the embedding $(P,Q)$ of the pair $(P, \Pi\cup\theta\Pi)$ is $$\chi(P, Q)= \frac{1}{2}\left(\|P\|+\|Q\|+\|PQ\|-2n\right),$$ where $2n$ is the number of elements in $S\cup S_{\theta}.$

We denote $c(P,Q)$ as the number of \textit{ orbits }that the group generated by $P$ and $Q.$ It is easy to see that if $a$ and $b$ are distinct and $a\equiv b (mod\ \Pi_i),$ then $$c(P,Q)+1\geq c(P_a,Q_a)\geq c(P,Q)$$ Moreover, if $a$ and $b$ belong to distinct cycles of $P$ or $a=bP,$  then $c(P_a,Q_a)= c(P,Q).$ Similarly, if $\theta a$ and $\theta b$ are distinct and $\theta a\equiv \theta b (mod\ \theta\Pi_i),$ then $$c(P_a,Q_a)+1\geq c(P_{a,\theta a},Q_{a,\theta a})\geq c(P_a,Q_a).$$ Again, if $\theta a$ and $\theta b$ belong to distinct cycles of $P_a$ or $\theta aP_a=\theta b,$  then $c(P_a,Q_a)= c(P_{a,\theta a},Q_{a,\theta a}).$

 The \textit{Euler-genus} %\textit{crosscap-number}
 of the embedding $(P,Q)$ of  non-orientable pair $(P, \Pi\cup\theta\Pi)$ is given by
 $$\gamma^E(P,Q)=\left\{
         \begin{array}{ll}
           2c(P,Q)-\chi(P, \Pi\cup\theta\Pi), & \hbox{if $(P, \Pi\cup\theta\Pi)$ is non-orientable}; \\
           c(P,Q)-\chi(P, \Pi\cup\theta\Pi), & \hbox{otherwise.}
         \end{array}
       \right.$$
  For $k\geq 0$,  let $\gamma^E(P,Q)(k)$  be the number of embeddings of  $(P, \Pi\cup\theta\Pi)$ that have Euler-genus $k.$ The\textit{ Euler-genus polynomial} of the permutation-bipartition pair $(P, \Pi\cup\theta\Pi)$ is

$$\mathcal{E}_{(P, \Pi\cup\theta\Pi)}(z)=\dsum_{i=0}^{\infty}\mathcal{E}_{(P, \Pi\cup\theta\Pi)}(k)z^k.$$

The next two corollaries are based on two lemmas in Subsection 2.2. We will refer to these two corollaries as the genus version of Walkup reduction in the sequel. The reason for this is that it is more convenient to represent genus polynomials of signed graphs, just like genus polynomials of graphs.

\begin{cor} \label{cor:3}Let  $(P, \Pi\cup\theta\Pi)$ be a permutation-bipartition pair with  $b\in S,$ $\{b\}\notin \Pi $ and $\theta b\in S_{\theta},$ $\{\theta b\}\notin \theta\Pi $, and let $c(a,b)=2(c(P,Q)-c(P_{a,\theta a},Q_{a,\theta a}).$ If $(P, \Pi\cup\theta\Pi)$ is non-orientable, then
	$$\mathcal{E}_{(P, \Pi\cup\theta\Pi)}(z)= \dsum_{\begin{array}{c}
a\equiv b (mod\ \Pi_i), \ a\neq b\\
\theta a\equiv \theta b (mod\ \theta\Pi_i), \ \theta a\neq \theta b
\end{array}} z^{\epsilon(a,\theta a)} \mathcal{E}_{(P_{a,\theta a}, \Pi_{a}\cup \theta \Pi_{a})}(z),$$
where  $$\epsilon(a,\theta a)=\left\{
         \begin{array}{ll}
           %2c(a,b)+1, & \hbox{if $a=bP,$ $\theta a P_a\neq\theta b,$ and $\theta a\equiv \theta b (mod P_a)$ }; \\
           2c(a,b)+2, & \hbox{if $a\neq bP,$ $\theta aP_a\neq\theta b,$ $ a\equiv b (mod P)$  and $\theta a\equiv \theta b (mod P_a)$ }; \\
           2c(a,b)+1, & \hbox{if $a\neq bP,$ $\theta aP_a\neq\theta b,$  and $\theta a\equiv \theta b (mod P_a)$ }; \\
           2c(a,b)+1, & \hbox{if $a\neq bP,$ $\theta aP_a\neq\theta b,$ and $ a\equiv b (mod P)$ }; \\
          %2c(a,b)+1, & \hbox{if $a\neq bP,$  $\theta a P_a=\theta b$ and $ a\equiv b (mod P)$ }; \\
                      0, & \hbox{otherwise.}
         \end{array}
       \right.
$$
Otherwise $$\mathcal{E}_{(P, \Pi\cup\theta\Pi)}(z)= \dsum_{\begin{array}{c}
a\equiv b (mod\ \Pi_i), \ a\neq b\\
\theta a\equiv \theta b (mod\ \theta\Pi_i), \ \theta a\neq \theta b
\end{array}} z^{\epsilon(a,\theta a)} \mathcal{E}_{(P_{a,\theta a}, \Pi_{a}\cup \theta \Pi_{a})}(z),$$
with  $$\epsilon(a,\theta a)=\left\{
         \begin{array}{ll}
           %c(a,b)+1, & \hbox{if $a=bP,$ $\theta a P_a\neq\theta b,$ and $\theta a\equiv \theta b (mod P_a)$ }; \\
           c(a,b)+2, & \hbox{if $a\neq bP,$ $\theta aP_a\neq\theta b,$ $ a\equiv b (mod P)$  and $\theta a\equiv \theta b (mod P_a)$ }; \\
           c(a,b)+1, & \hbox{if $a\neq bP,$ $\theta aP_a\neq\theta b,$  and $\theta a\equiv \theta b (mod P_a)$ }; \\
           c(a,b)+1, & \hbox{if $a\neq bP,$ $\theta aP_a\neq\theta b,$ and $ a\equiv b (mod P)$ }; \\
          %c(a,b)+1, & \hbox{if $a\neq bP,$  $\theta a P_a=\theta b$ and $ a\equiv b (mod P)$ }; \\
                      0, & \hbox{otherwise.}
         \end{array}
       \right.
$$
\end{cor}
\begin{proof}We only give a proof for the non-orientable case. To prove the theorem, it suffices to show that $$\gamma^E(P,Q)=\gamma^E(P_{a,\theta a}, Q_{a,\theta a})+\epsilon(a,\theta a).$$
Recall that $\{b\}\notin \Pi,$ and $\{\theta b\}\notin \theta\Pi,$ in this case, we always have $\|Q\|-\|Q_{a,\theta a}\|=0$. The following four cases are discussed.
\begin{description}
  \item[Case 1] If $a=bP,$ and  $\theta aP_a=\theta b,$  then $\|P\|-\|P_{a,\theta a}\|=0$, and $c(P,Q)=c(P_a,Q_a)=c(P_{a,\theta a},Q_{a,\theta a}) $ By Lemma \ref{lem:S1}, $\|PQ\|-\|P_{a,\theta a}Q_{a,\theta a}\|=2.$ Thus,
\begin{align*}&\gamma^E(P,Q)-\gamma^E(P_{a,\theta a}, Q_{a,\theta a})\\=&2(c(P,Q)-c(P_{a,\theta a},Q_{a,\theta a}))-\frac{1}{2}(\|P\|-\|P_{a,\theta a}\|)
\\&-\frac{1}{2}(\|Q\|-\|Q_{a,\theta a}\|)-\frac{1}{2}(\|PQ\|-\|P_{a,\theta a}Q_{a,\theta a}\|)\\&+\frac{1}{2}(2n-2(n-1))
\\=&0-0-0-1+1=0.
\end{align*}
  \item[Case 2] If $a=bP,$ and  $\theta a P_a\neq\theta b,$ it's impossible.% then $\|PQ\|-\|P_{a,\theta a}Q_{a,\theta a}\|=1$ by Lemma \ref{lem:S1}. If $\theta a$ and $\theta b$ belong to the same cycle of $P_a,$ then $\|P\|=\|P_{a,\theta a}\|-1,$ and\begin{align*}&\gamma^E(P,Q)-\gamma^E(P_{a,\theta a}, Q_{a,\theta a})\\=&2(c(P,Q)-c(P_{a,\theta a},Q_{a,\theta a}))-\frac{1}{2}(\|P\|-\|P_{a,\theta a}\|)\\&-\frac{1}{2}(\|Q\|-\|Q_{a,\theta a}\|)-\frac{1}{2}(\|PQ\|-\|P_{a,\theta a}Q_{a,\theta a}\|)\\&+\frac{1}{2}(2n-2(n-1))\\=&2(c(P,Q)-c(P_{a,\theta a},Q_{a,\theta a}))+\frac{1}{2}-0-\frac{1}{2}+1\\=&2(c(P,Q)-c(P_{a,\theta a},Q_{a,\theta a}))+1.\end{align*}
%Otherwise $\theta a$ and $\theta b$ belong to the distinct cycles of $P_a,$ then $\|P\|=\|P_{a,\theta a}\|+1,$ and\begin{align*}&\gamma^E(P,Q)-\gamma^E(P_{a,\theta a}, Q_{a,\theta a})\\=&0-\frac{1}{2}-0-\frac{1}{2}+1\\=&0.\end{align*}
  \item[Case 3]If $a\neq bP,$ and  $\theta aP_a\neq\theta b,$  then $\|PQ\|-\|P_{a,\theta a}Q_{a,\theta a}\|=0$ by Lemma \ref{lem:S1}. If
  both $a$ and $b$, and $\theta a$ and $\theta b$ belong to the distinct cycles of $P$ and $P_a$, respectively, then $\|P\|=\|P_{a,\theta a}\|+2,$
    and $c(P,Q)=c(P_{a,\theta a},Q_{a,\theta a}),$ then
    \begin{align*}&\gamma^E(P,Q)-\gamma^E(P_{a,\theta a}, Q_{a,\theta a})\\=&2(c(P,Q)-c(P_{a,\theta a},Q_{a,\theta a}))-\frac{1}{2}(\|P\|-\|P_{a,\theta a}\|)
\\&-\frac{1}{2}(\|Q\|-\|Q_{a,\theta a}\|)-\frac{1}{2}(\|PQ\|-\|P_{a,\theta a}Q_{a,\theta a}\|)\\&+\frac{1}{2}(2n-2(n-1))
\\=&0-1-0-0+1\\
=&0.
\end{align*}
    %\begin{align*}\gamma^E(P,Q)-\gamma^E(P_{a,\theta a}, Q_{a,\theta a})=0.\end{align*}
If  both $a$ and $b$, and $\theta a$ and $\theta b$ belong to the same cycle of $P$ and $P_a$, respectively, then $\|P\|=\|P_{a,\theta a}\|-2,$
    and $c(P,Q)=c(P_{a,\theta a},Q_{a,\theta a}),$ and \begin{align*}\gamma^E(P,Q)-\gamma^E(P_{a,\theta a}, Q_{a,\theta a})=&2(c(P,Q)-c(P_{a,\theta a},Q_{a,\theta a}))+2.
\end{align*}
In other cases, we have
\begin{align*}\gamma^E(P,Q)-\gamma^E(P_{a,\theta a}, Q_{a,\theta a})=&2(c(P,Q)-c(P_{a,\theta a},Q_{a,\theta a}))+1.\end{align*}

  \item[Case 4]If $a\neq bP,$ and  $\theta a P_a=\theta b,$  this case is impossible.%  then $\|PQ\|-\|P_{a,\theta a}Q_{a,\theta a}\|=1$ by Lemma \ref{lem:S1}. If $a$ and $b$ belong to the same cycle of $P,$ then $\|P\|=\|P_{a,\theta a}\|-1,$ thus $$\gamma^E(P,Q)-\gamma^E(P_{a,\theta a}, Q_{a,\theta a})=2(c(P,Q)-c(P_{a,\theta a},Q_{a,\theta a}))+1.$$Otherwise $a$ and $b$ belong to the distinct cycles of $P,$ then $\|P\|=\|P_{a,\theta a}\|+1,$ and\begin{align*}\gamma^E(P,Q)-\gamma^E(P_{a,\theta a}, Q_{a,\theta a})=0.\end{align*}

\end{description}
The result follows.

\end{proof}

\begin{cor} \label{cor:4}Let  $(P, \Pi\cup\theta\Pi)$ be a permutation-bipartition pair with $\{b\}\in \Pi, \{\theta b\}\in \theta\Pi ,$ and let $(P_{b,\theta b}, \Pi_{b}\cup \theta \Pi_{b}), = (P, \Pi\cup\theta\Pi)|_{\theta b\rightarrow \theta b}^{b\rightarrow b}$. If $(P, \Pi\cup\theta\Pi)$ is non-orientable, then	$$\mathcal{E}_{(P, \Pi\cup\theta\Pi)}(z)=z^{\epsilon(b,\theta b)}   \mathcal{E}_{(P_{b,\theta b}, \Pi_{b}\cup \theta \Pi_{b})}(z)$$
 where $$\epsilon(b,\theta b)=\left\{
         \begin{array}{ll}
           2, & \hbox{if $b= bP,$ and $\theta bP_b=\theta b$ }; \\
           %1, & \hbox{if $b= bP,$ and $\theta bP_b\neq\theta b$ }; \\
           %1, & \hbox{if $b\neq bP,$ and $\theta bP_b=\theta b$ }; \\
           0, & \hbox{if $b\neq bP,$ and  $\theta b\neq \theta bP_b.$}
         \end{array}
       \right.
$$
Otherwise we have $$\mathcal{E}_{(P, \Pi\cup\theta\Pi)}(z)=\mathcal{E}_{(P_{b,\theta b}, \Pi_{b}\cup \theta \Pi_{b})}(z).$$
\end{cor}
\begin{proof}Again, we only prove the theorem for the non-orientable case. Since $\{b\}\in \Pi,$ and $\{\theta b\}\in \theta\Pi,$ we have $(\|Q\|-\|Q_{b,\theta b}\|)=2.$ The following four cases are discussed.
%\begin{description}\item[Case 1]

If $b=bP,$ and  $\theta b=\theta bP_b,$ then $c(P,Q)=c(P_{b,\theta b},Q_{b,\theta b})+2,$ and $\|P\|=\|P_{b,\theta b}\|+2$. By Lemma \ref{lem:S2}, we have
\begin{align*}&\gamma^E(P,Q)-\gamma^E(P_{b,\theta b}, Q_{b,\theta b})\\=&2(c(P,Q)-c(P_{b,\theta b},Q_{b,\theta b}))-\frac{1}{2}(\|P\|-\|P_{b,\theta b}\|)
\\&-\frac{1}{2}(\|Q\|-\|Q_{b,\theta b}\|)-\frac{1}{2}(\|PQ\|-\|P_{b,\theta b}Q_{b,\theta b}\|)\\&+\frac{1}{2}(2n-2(n-1))
\\=&4-1-1-1+1=2
\end{align*}
  %\item[Case 2]If $b=bP,$ and  $\theta b\neq \theta bP_b,$ then $c(P,Q)=c(P_{b,\theta b},Q_{b,\theta b})+1,$ and $\|P\|=\|P_{b,\theta b}\|+1$. By Lemma \ref{lem:S2},\begin{align*}\gamma^E(P,Q)-\gamma^E(P_{b,\theta b}, Q_{b,\theta b})=2-\frac{1}{2}-1-\frac{1}{2}+1=1\end{align*}
  %\item[Case 3] If $b\neq bP,$ and  $\theta b=\theta bP_b,$ again, we have $c(P,Q)=c(P_{b,\theta b},Q_{b,\theta b})+1,$ and $\|P\|=\|P_{b,\theta b}\|+1$. By Lemma \ref{lem:S2},\begin{align*}\gamma^E(P,Q)-\gamma^E(P_{b,\theta b}, Q_{b,\theta b})=2-\frac{1}{2}-1-\frac{1}{2}+1=1\end{align*}
  %\item[Case 4] 
  
  If $b\neq bP,$ and  $\theta b\neq \theta bP_b,$ then $c(P,Q)=c(P_{b,\theta b},Q_{b,\theta b}),$ and $\|P\|=\|P_{b,\theta b}\|$. By Lemma \ref{lem:S2}, \begin{align*}\gamma^E(P,Q)-\gamma^E(P_{b,\theta b}, Q_{b,\theta b})=0-0-1+0+1=0
\end{align*}
%\end{description}
From the discussions above, we have proved that $\gamma^E(P,Q)=\gamma^E(P_{b,\theta b}, Q_{b,\theta b})+\epsilon(b,\theta b),$ the result follows.
\end{proof}

The two corollaries above imply the genus version of the Walkup reduction of Subsection \ref{region}. Hence, we can associate with each pair $(P_{\Sigma_n} ,\Pi_{\Sigma_n}\cup\theta{\Pi}_{\Sigma_n} )$ a \textit{genus reduction diagram }$\mathfrak{G}_{(P, \Pi\cup\theta\Pi)}$ which differs from the region reduction diagram only in that the edge labels $\delta_{a,bP}+\delta_{\theta aP_a,\theta b}$ and $\delta_{b,bP}+\delta_{\theta b,\theta bP_b}$ are replaced with $\epsilon(a,\theta a)$ and $\epsilon(b,\theta b)$ respectively. The following theorem follows.

\begin{thm}\label{genus}
The embeddings of the permutation-bipartition pair $(P, \Pi\cup\theta\Pi)$ are in a one-to-one corresponding with the directed path of $\mathfrak{G}_{(P, \Pi\cup\theta\Pi)}.$ This correspondence is such that the Euler-genus of the embedding $(P, \Pi\cup\theta\Pi)$ is the sum of the weights along the corresponding path.
\end{thm}

\begin{figure}[h]
\centering
\includegraphics[width=2.68cm]{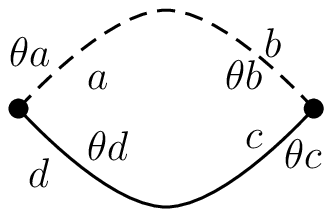}\\

\includegraphics[width=10cm]{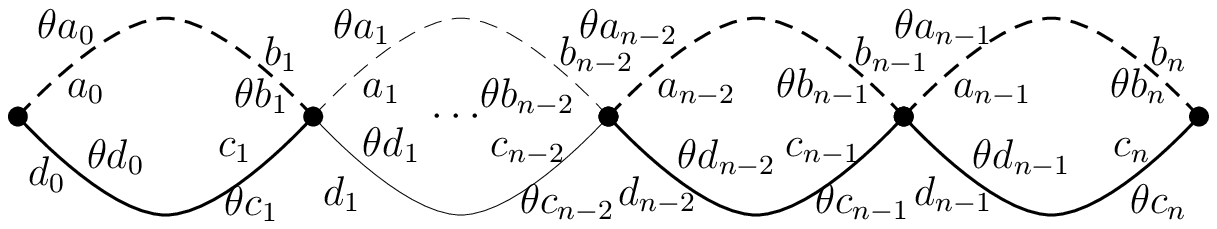}
\caption{The signed graph $\Sigma=(C_2,\sigma)$ with one negative edge and one positive edge, and the $\Sigma$-linear graph $\Sigma_n.$}\label{fig:H}
\end{figure}
\vskip 1cm
Let $(C_2,\sigma)$ be the signed graph of Figure \ref{fig:H}. The signed graph $\Sigma_n$ is obtained by successively amalgamating $n$ copies of $(C_2,\sigma)$ as shown in  Figure \ref{fig:H}.  Let $\Sigma_n^{'}$ denote the permutation-bipartition
pair obtained from $(P_{\Sigma_n} ,\Pi_{\Sigma_n}\cup\theta{\Pi}_{\Sigma_n} )$ by replacing the four transpositions $$(b_n\theta{a}_{n-1})(\theta b_n {a}_{n-1})(c_n{d}_{n-1})(\theta c_n\theta{d}_{n-1})$$ with
$$(b_{n}\theta{a}_{n-1}\theta d_{n-1}\theta c_{n})(b_{n}\theta{a}_{n-1}\theta d_{n-1}\theta c_{n}).$$ Again, let $\Sigma_n^{''}$ denote the permutation-bipartition pair obtained from $(P_{\Sigma_n} ,\Pi_{\Sigma_n}\cup\theta{\Pi}_{\Sigma_n} )$ by replacing the four transpositions $$(b_n\theta{a}_{n-1})(\theta b_n {a}_{n-1})(c_n{d}_{n-1})(\theta c_n\theta{d}_{n-1})$$ with
$$(c_{n}d_{n-1}b_{n}\theta{a}_{n-1})(\theta c_{n}a_{n-1} \theta b_{n}\theta{d}_{n-1}).$$ Apply the genus version of the Walkup reduction to $(P_{\Sigma_n} ,\Pi_{\Sigma_n}\cup\theta{\Pi}_{\Sigma_n} )$ so as to eliminate
all the bits in the last copy of $(C_2,\sigma)$. Figure \ref{fig:redu-1}, Figure \ref{fig:redu-2} and Figure \ref{fig:redu-3} illustrate the process, it follows that

 \begin{eqnarray}\label{cp-reduction}
 \mathcal{E}_{\Sigma_n}(z)&=&4z\mathcal{E}_{\Sigma_{n-1}}(z)+2\mathcal{E}_{\Sigma_{n-1}^{'}}(z) \\
 \mathcal{E}_{\Sigma_{n}^{'}}(z)&=&2z\mathcal{E}_{\Sigma_{n-1}}(z)+4z\mathcal{E}_{\Sigma_{n-1}^{''}}(z)\\
 \mathcal{E}_{\Sigma_{n}^{''}}(z)&=&6z^2\mathcal{E}_{\Sigma_{n-1}}(z)
 \end{eqnarray}
Actually, Equation (3.1), Equation (3.2) and Equation (3.2) follow from  Figure \ref{fig:redu-1}, Figure \ref{fig:redu-2} and Figure \ref{fig:redu-1}, respectively.
 Let $$V_n=\left(
             \begin{array}{c}
                \mathcal{E}_{\Sigma_n}(z) \\
                \mathcal{E}_{\Sigma_{n}^{'}}(z) \\
               \mathcal{E}_{\Sigma_{n}^{''}}(z) \\
             \end{array}
           \right)
 $$

So the equations above is equivalent to the follow equation $$V_n(z)=\left(
                                                                    \begin{array}{ccc}
                                                                      4z & 2 & 0 \\
                                                                      4z & 0 & 2z \\
                                                                      6z^2 & 0 & 0 \\
                                                                    \end{array}
                                                                  \right)V_{n-1}(z),
$$
for $n\geq 2.$

Since $(P_{\Sigma_1} ,\Pi_{\Sigma_1}\cup\theta{\Pi}_{\Sigma_1} )=\{(b_1\theta{a}_{0})(\theta b_1 {a}_{0})(c_1{d}_{0})(\theta c_1\theta{d}_{0}),\{\{b_1,c_1\},\{a_0,d_0\}, $ $\{\theta b_1,\theta c_1\}, \{\theta{a}_{0},\theta{d}_{0}\}\}$,  $(P_{\Sigma_1^{'}} ,\Pi_{\Sigma_1^{'}}\cup\theta{\Pi}_{\Sigma_1^{'}} )=\{(b_{1}\theta{a}_{0}\theta d_{0}\theta c_{1})(\theta b_{1}{c}_{1} d_{0}a_{0}),\{\{b_1,c_1\},$ $\{a_0,d_0\},\{\theta b_1,\theta c_1\}, \{\theta{a}_{0},\theta{d}_{0}\}\}$, and  $(P_{\Sigma_1^{''}} ,\Pi_{\Sigma_1^{''}}\cup\theta{\Pi}_{\Sigma_1^{''}} )=\{(c_{1}d_{0}b_{1}\theta{a}_{0})(\theta c_{1}a_{0} \theta b_{1}\theta{d}_{0}),$ $\{\{b_1,c_1\},\{a_0,d_0\},\{\theta b_1,\theta c_1\}, \{\theta{a}_{0},\theta{d}_{0}\}\}$, it's follows that

$$V_1(z)=\left(
             \begin{array}{c}
                z \\
                z \\
               z^2\\
             \end{array}
           \right)
 $$

Let's list the values of $V_n(z)$ for $n=2,3,4,$ and $5,$

 \begin{eqnarray*}
V_2(z)&=&\left(
\begin{array}{c}
 4 z^2+2 z \\
 2 z^3+4 z^2 \\
 6 z^3 \\
\end{array}
\right),
\\
V_3(z)&=&\left(
\begin{array}{c}
 20 z^3+16 z^2 \\
 12 z^4+16 z^3+8 z^2 \\
 24 z^4+12 z^3 \\
\end{array}
\right),\\
V_4(z)&=&\left(
\begin{array}{c}
 104 z^4+96 z^3+16 z^2 \\
 48 z^5+104 z^4+64 z^3 \\
 120 z^5+96 z^4 \\
\end{array}
\right),\\
V_{5}(z)&=&\left(
\begin{array}{c}
 512 z^5+592 z^4+192 z^3 \\
 240 z^6+608 z^5+384 z^4+64 z^3 \\
 624 z^6+576 z^5+96 z^4 \\
\end{array}
\right).
 \end{eqnarray*}

\begin{figure}[h]
\centering
\includegraphics[width=14cm,height=14cm]{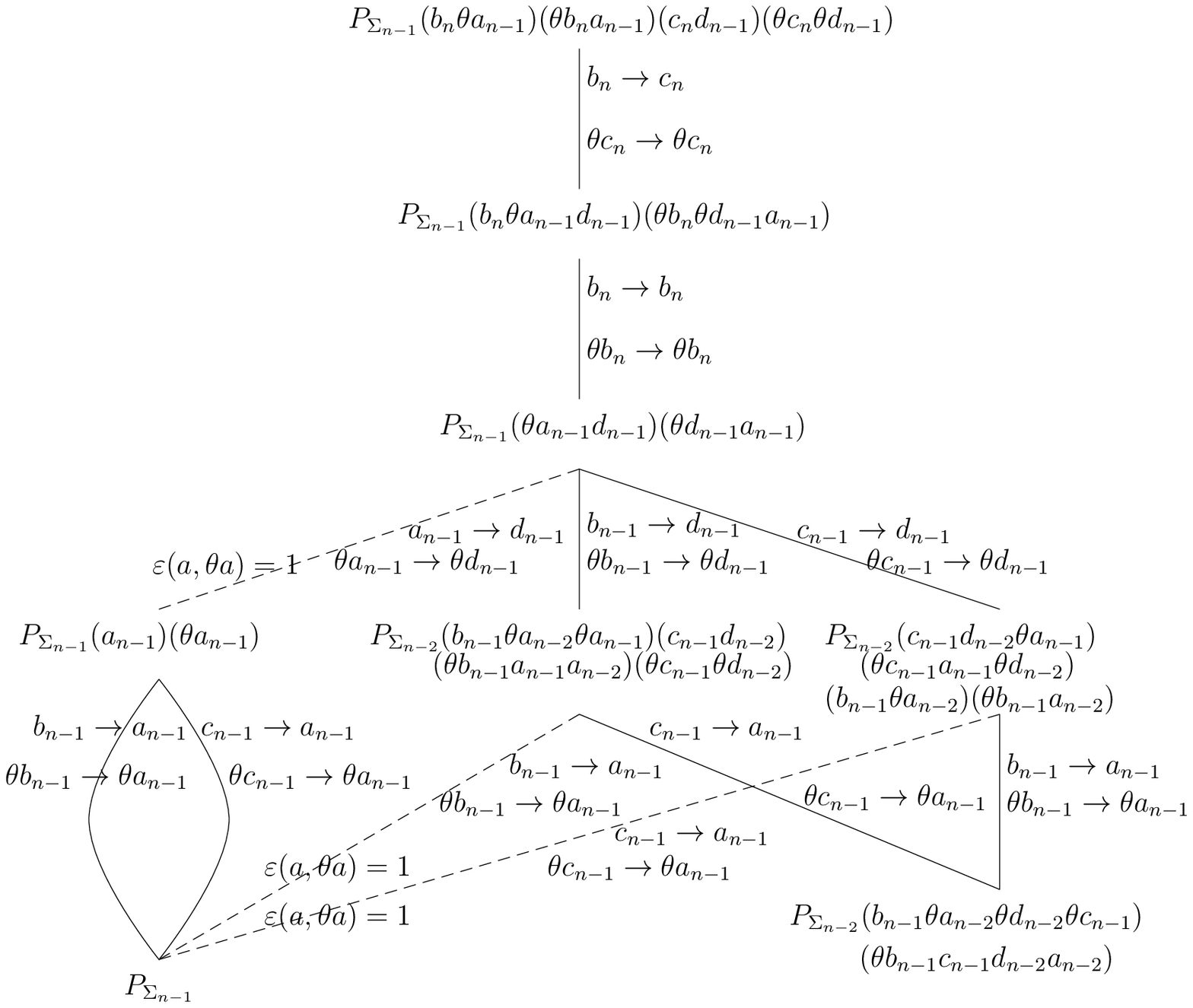}
\caption{}\label{fig:redu-1}
\end{figure}

\begin{figure}[h]
\centering
\includegraphics[width=14cm,height=14cm]{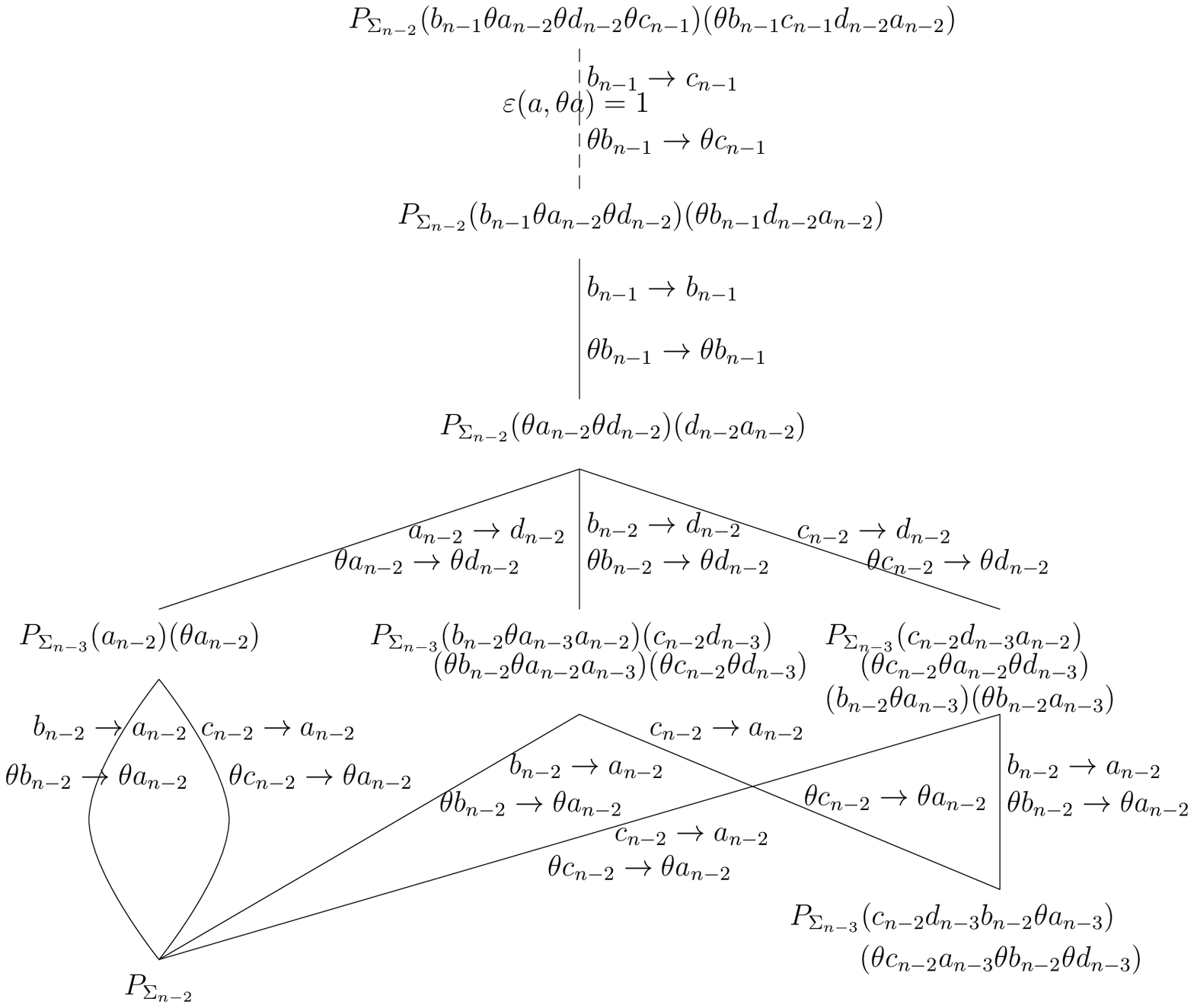}
\caption{}\label{fig:redu-2}
\end{figure}
\begin{figure}[h]
\centering
\includegraphics[width=11cm,height=13cm]{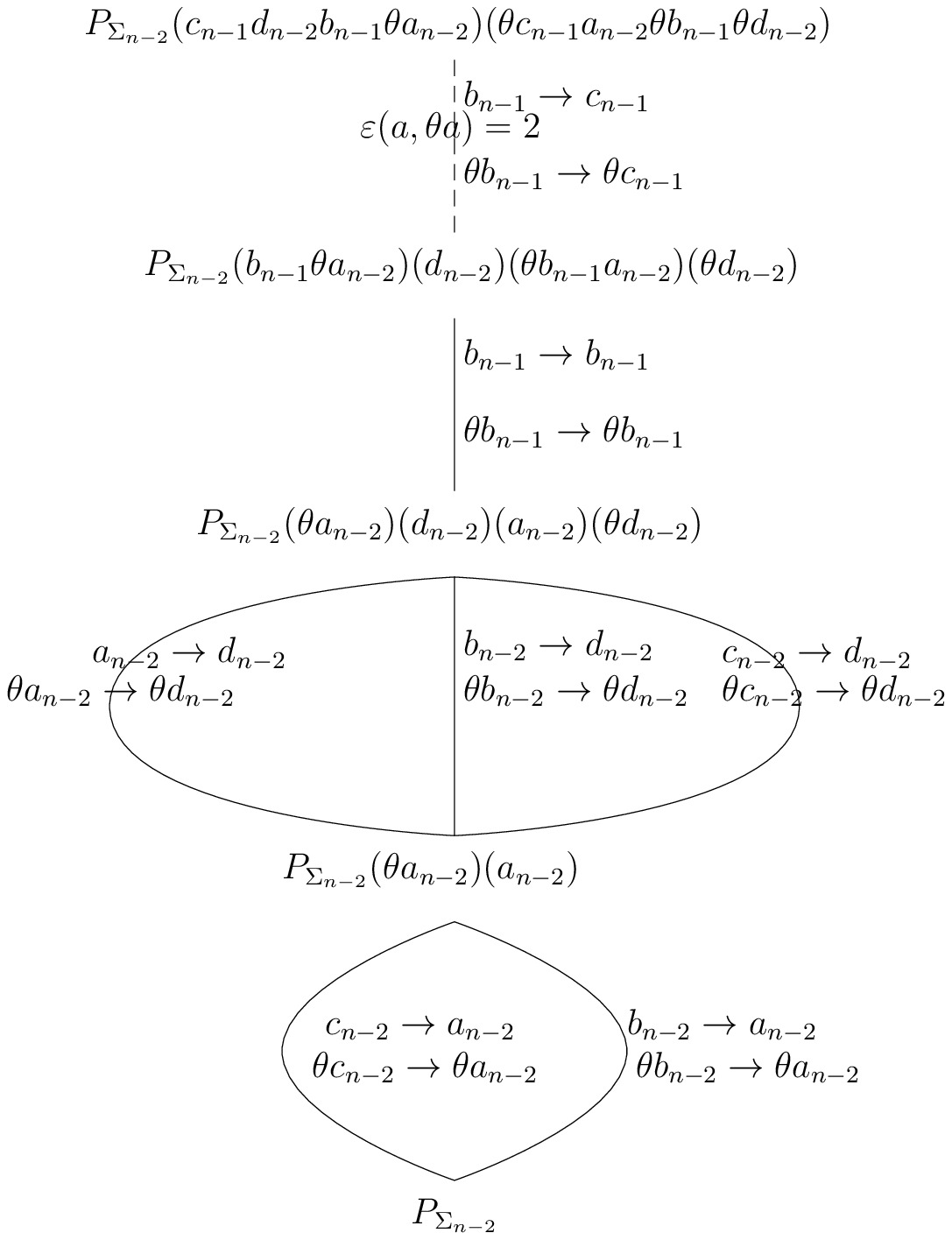}
\caption{}\label{fig:redu-3}
\end{figure}

We give a review of $\Sigma$-linear signed graph families which is a generalization of $H$-linear graph families which was introduced by Stahl in \cite{Sta91a}. Suppose $\Sigma$ is a  connected signed graph. Let $U=\{u_1,u_2,\ldots,u_s\}$ and $V=\{v_1,v_2,\ldots,v_s\}$ be disjoint subsets of $V_\Sigma$. For $i=1,2,\ldots$ , let $\Gamma_i$ be a copy of $\Sigma$ and  let $f_i:\Sigma\rightarrow \Gamma_i$ be an isomorphism. For each $i\geq1$ and $1\leq j\leq s$, we let $u_{i,j}=f_i(u_j)$ and $v_{i,j}=f_i(v_j)$.  An  \textit{{$\Sigma$-linear family}} of sigened graphs $\mathbf{\mathcal{S}} =\{\Sigma_n\}_{n=1}^{\infty}$ is constructed recursively  as follows: (1) $\Sigma_1=\Gamma_1$. (2) For  $j=1,2,\ldots,s,$ $\Sigma_n$ is constructed from $\Sigma_{n-1}$ and $\Gamma_n$ by amalgamating the vertex $v_{n-1,j}$ of $\Sigma_{n-1}$ with the vertex $u_{{n},{j}}$ of $ \Gamma_{n}$.

The procedure above generalizes to the following analog of proposition 5.2 of \cite{Sta97}.

\begin{prop} Let $\mathbf{\mathcal{S}} =\{\Sigma_n\}_{n=1}^{\infty}$  be an $\Sigma$-linear family of signed graphs. Then there
exist a positive integer $k$, a $k\times k$ transfer matrix $M(z)$ and a column $k$-vector $v(z)$, with integer coefficients, such that the first entry of $M(z)^nV(z)$ is $\mathcal{E}_{\Sigma_n}(z)$ .
\end{prop}
Again, by using the similar technique of the proof of Theorem 3.2 of \cite{CGMT20}, we have the following result.
\begin{thm} \label{linear:genus}
Let $\mathbf{\mathcal{S}} =\{\Sigma_n\}_{n=1}^{\infty}$  be an $\Sigma$-linear family of signed graphs. Then there exist a positive integer $k$ and polynomials $a_1(z), a_2(z), \ldots, a_k(z)$ with integer coefficients, such that the Euler-genus polynomial $\mathcal{E}_{\Sigma_n}(z)$ satisfies the $k\uth$-order homogeneous recursive equation
$$
\mathcal{E}_{\Sigma_n}(z)= a_1(z)\mathcal{E}_{\Sigma_{n-1}}(z)(z)+a_2(z)\mathcal{E}_{\Sigma_{n-2}}+\cdots+a_k(z)\mathcal{E}_{\Sigma_{n-k}}
$$
\noindent with the initial conditions
$\mathcal{E}_{\Sigma_1}(z),\mathcal{E}_{\Sigma_2}(z),\ldots,\mathcal{E}_{\Sigma_k(z)}.$
\end{thm}
\smallskip

Since a graph is a signed graph, Theorem \ref{linear:genus} can be seen as a generalization of Theorem 3.2 of \cite{CGMT20}.

It's obvious that we can also use the permutation-bipartition pairs to calculate the crosscap-number distribution of a graph,  i.e., the distribution
of the embeddings of a graph in the nonorientable surfaces, which was introduced in \cite{CGR94} by Chen, Gross and Rieper.  In \cite{CGR94}, they  demonstrated
how to use the rank of Mohar's overlap matrix \cite{Moh89} to calculate
the crosscap-number distribution. For combinatorial approach to the calculation of the crosscap-number distributions of graphs, see a recent paper of \cite{CG18}. We may refer the reader to \cite{GF87,FGS89,GRT89,Sta91b,Gro11a,Gro14,GT21} for calculating genus distributions of a graph and its history.

\section{The expected genus}

 Let $\Sigma$ be a connected signed graph, the expected genus (or average genus ) of $\Sigma$ is given by $$\gamma_{E} (G)=\frac{\sum_{i\geq0}i\mathcal{E}_{\Sigma}(i)}{\prod_{v\in V(\Sigma)}(d_v-1)!}.$$

\noindent In other words, we have $\gamma_{E} (\Sigma)=\frac{\mathcal{E}_\Sigma(1)^{'}}{\mathcal{E}_\Sigma(1)}.$ We shall frequently use this property between expected genus and genus polynomial in the following discussion.

In \cite{Whi94}, White calculated the exact values for expected genus of  ladders and cobblestone paths. Tesar \cite{Tes00} calculated expected genus for the Ringel ladder graph.  Recall that ladders and cobblestone paths are two classes of linear graph families. Using permutation-partition pairs and Peron-Frobenius theory of stochastic matrices, Stahl \cite{Sta91a} gave the asymptotic result for expected genus of linear graph families. By using a different approach, we now slightly generalize Stahl's result to signed graphs.

\begin{prop}\label{1}  Let $D=\dsum_{i=1}^{k}m_{ij}(1)$, then $\frac{1}{D} M(1)$ is a  stochastic matrix.
\end{prop}

 The stochastic matrix $\frac{1}{D} M(1)$ is called the \textit{associated matrix} of $\Sigma_n.$ We also call $D$ \textit{stochastic constant} of $\Sigma_n.$

\begin{lemma}\label{lemma:nonhomo}Let $\mathbf{\mathcal{S}} =\{\Sigma_n\}_{n=1}^{\infty}$  be an $\Sigma$-linear family of signed graphs. Then there exist a positive integer $k$ and rational numbers $c_1,c_2,\ldots,c_k,d$ such that the expected genus $\gamma_{E}(G_n^{\circ})$ satisfies the  $k\uth$-order nonhomogeneous recurrence relation
\begin{eqnarray}\label{Aver:nonhomo}
\gamma_{E}(\Sigma_n)&=& c_1\gamma_{E}(\Sigma_{n-1})+c_2\gamma_{E}(\Sigma_{n-2})+\cdots+c_k\gamma_{E}(\Sigma_{n-k})+d
\end{eqnarray}
with initial values $\gamma_{E}(\Sigma_1),\gamma_{E}(\Sigma_2),\ldots,\gamma_{E}(\Sigma_k)$ and $c_1+c_2+\cdots+c_k=1.$
\end{lemma}
\begin{proof}By Theorem \ref{linear:genus}, there exist a positive integer $k$ and polynomials $b_1(z),b_2(z),\ldots,b_k(z)$ with integer coefficients, such that the Euler-genus polynomial  $\mathcal{E}_{\Sigma_n}(z)$ satisfies the $k\uth$-order homogeneous recurrence relation
\begin{eqnarray} \label{Eulerequation}
\mathcal{E}_{\Sigma_n}(z) &=& b_1(z)\mathcal{E}_{\Sigma_{n-1}}(z) \+ b_2(z)\mathcal{E}_{\Sigma_{n-2}}(z) \+ \cdots \+ b_k(z)\mathcal{E}_{\Sigma_{n-k}}(z) ,
\end{eqnarray}
with initial conditions
$\mathcal{E}_{\Sigma_1}(z),\, \mathcal{E}_{\Sigma_2}(z),\, \ldots,\,  \mathcal{E}_{\Sigma_k}(z).$

 Taking derivative of the both sides of the recurrence equation (\ref{Eulerequation}) with respect to $x$. Then letting $x=1$ and multiplying both sides by $\frac{1}{\mathcal{E}_{\Sigma_n}(1)},$ we obtain

 \begin{eqnarray} \label{Aver:1}
\gamma_{E}(\Sigma_n)&=& c_1\gamma_{E}(\Sigma_{n-1})+c_2\gamma_{E}(\Sigma_{n-2})+\cdots+c_k\gamma_{E}(\Sigma_{n-k})+d\notag,
\end{eqnarray} where $c_i=b_i(1)\left(\frac{\mathcal{E}_{\Sigma_{n-1}}(1)}{\mathcal{E}_{\Sigma_{n}}(1)}\right)^i,$ for $i=1,2,\ldots,k,$ and $d=\frac{\dsum_{i=1}^{k} b_i^{'}(1) \mathcal{E}_{\Sigma_i}(1)}{\mathcal{E}_{\Sigma_{n}}(1)}.$

Note that ${D}=\frac{\mathcal{E}_{\Sigma_{n}}(1)}{\mathcal{E}_{\Sigma_{n-1}}(1)}.$ We now proceed to prove that $$c_1+c_2+\cdots+c_k=1,$$
which is equivalent to show $$b_1(1)D^{k-1}+b_2(1)D^{k-2}+\cdots +b_k(1)=D^k.$$

Let $M(z)$ be the transfer matrix of $\Sigma_{n}.$
We have
\begin{equation}\label{kernel}
M(z)^{k} - b_1(z)M(z)^{k-1} - \cdots - b_{k-1}(z)M(z) - b_k(z) \= 0,
\end{equation}
where
\begin{equation*}
\lambda^{k} - b_1(z)\lambda^{k-1} - \cdots-b_{k-1}(z)\lambda  -b_k(z)
\end{equation*}
is the characteristic polynomial of the transfer matrix $M(z).$

Since $\frac{1}{D}M(1)$ is a stochastic matrix and  each stochastic matrix has an eigenvalue $1,$ then $M(1)$ has an eigenvalue $D.$

Thus \begin{equation*}
D^{k} - b_1(1)D^{k-1} - \cdots-b_{k-1}(1)D  -b_k(1)=0.
\end{equation*}
The theorem follows.
\end{proof}

\begin{lemma}\label{part:solution}Suppose $A$ is a constant, then $Ad+\frac{d(n+1)}{1+c_2+2c_3+\ldots+(k-1)c_k}$ is a solution for the  $k\uth$-order nonhomogeneous recurrence equation (\ref{Aver:nonhomo}).
\end{lemma}
\begin{proof}To find a particular solution of the  nonhomogeneous recurrence equation (\ref{Aver:nonhomo}), we try a solution of the form $$\gamma_{E}(\Sigma_n)=Ad+Bd(n+1),$$
where  $A,B$ are  constant coefficients to be determined.

\noindent We now have

\begin{eqnarray*}Ad+Bd(n+1)&=&c_1(Ad+Bdn)+c_2(Ad+Bd(n-1))\\
&&+\cdots+c_k\left(Ad+Bd(n-k+1)\right)+d\\
&=&Ad\dsum_{i=1}^{k}c_i+Bdn\dsum_{i=1}^{k}c_i-Bd\dsum_{i=2}^{k}c_i(i-1)+d\\
&=&Ad+Bdn-Bd\dsum_{i=2}^{k}c_i(i-1)+d
\end{eqnarray*}
Thus, $$B=\frac{1}{1+c_2+2c_3+\ldots+(k-1)c_k}.$$

The result follows.
\end{proof}

\begin{remark}From the proof of Lemma \ref{part:solution}, the constant $A$ can be determined by the initial values of Equation (\ref{Aver:nonhomo}).
\end{remark}

The following property of primitive stochastic matrix can be found in Proposition 9.2 in \cite{Beh}.
\begin{prop}\cite{Beh}\label{eige}Every eigenvalue $\lambda$ of a stochastic matrix $M$  satisfies $|\lambda|\leq1.$ Furthermore, if the stochastic matrix $M$ is primitive, then all other eigenvalues of modulus are less than $1$, and algebraic multiplicity of $1$ is one.
\end{prop}

For $n>1$, the graph $\Sigma_n$ is \textit{regular}, if the associated matrix $\frac{1}{D} M(1)$ is primitive or all eigenvalues of $M(1)$ are real.

\begin{lemma}\label{bigO} We have $\gamma_{E}(\Sigma_n)=O(n).$
\end{lemma}
\begin{proof} From the definition of expected genus and linear graph families with spiders, we have \begin{align*} &\gamma_{E}(\Sigma_n)<\beta(\Sigma_n)\\&\leq\beta(\Sigma_{n-1})+(\beta(H)+s-1)\\&= n(\beta(H)+s-1)+2s\  max\{\beta(J_1),\cdots,\beta(J_s),\beta(\ov{J}_1),\cdots,\beta(\ov{J}_s)\}\\&=O(n),\end{align*} the result follows.
\end{proof}

\begin{thm}\label{Aver:orienspider}
Let  $\mathbf{\mathcal{S}} =\{\Sigma_n\}_{n=1}^{\infty}$  be an $\Sigma$-linear family of signed graphs.  Then there exist a constant $C,$ an integer $p,$ and a periodic sequence $\{B_n\}_{n=1}^{\infty}$  with $B_m=B_n$ whenever $m\equiv n $ (mod $p$), \begin{eqnarray}\label{Genusequation}
\gamma_{E}(\Sigma_n)=  Cn+B_n+o(1).
\end{eqnarray}
Furthermore, if $\Sigma_n$ is \textit{regular}, then \begin{eqnarray}\label{Genusequation}
\gamma_{E}(\Sigma_n)=  Cn+B+o(1),
\end{eqnarray}
where $B$ is a constant.
\end{thm}

\begin{proof}Supposed that the associated homogeneous equation of (\ref{Aver:nonhomo}) is
\begin{eqnarray}\label{Genusequation:homo}
\gamma_{E}(\Sigma_n)&=& c_1\gamma_{E}(\Sigma_{n-1})+c_2\gamma_{E}(\Sigma_{n-2})+\cdots+c_k\gamma_{E}(\Sigma_{n-k}),
\end{eqnarray}
where $c_1+c_2+\cdots+c_k=1.$

\noindent Let $$\lambda^k-c_1\lambda^{k-1}-c_2\lambda^{k-2}-\cdots-c_k=0$$ be the characteristic equation of equation (\ref{Genusequation:homo}). Suppose that the distinct characteristic roots are $\lambda_1,\lambda_2,\ldots,\lambda_r$ with multiplicities $m_1,m_2,\ldots,m_r$ with $\dsum_{i=1}^{r}m_i=k,$ respectively.
Thus, the general solution of equation (\ref{Genusequation:homo}) is given by $$\dsum_{i=1}^{r}\lambda_i^n(a_{i,0}+a_{i,1}n+\cdots+a_{i,m_{i}-1}n^{m_i-1}).$$

From the proof in Lemma \ref{lemma:nonhomo}, we know that $\lambda_1,\lambda_2,\ldots,\lambda_r$ are eigenvalues of the stochastic matrix $\frac{1}{D}M(1).$ Recall that the general solution to the nonhomogeneous  equation (\ref{Aver:nonhomo}) is the general solution equals the general solution to the associated homogeneous  equation plus any particular solution to the nonhomogeneous equation (\ref{Aver:nonhomo}). Thus,
\begin{eqnarray*}\gamma_{E}(\Sigma_n)&=&\dsum_{i=1}^{r}\lambda_i^n(a_{i,0}+a_{i,1}n+\cdots+a_{i,m_{i}-1}n^{m_i-1}).\\
&&+Ad+\frac{d(n+1)}{1+c_2+2c_3+\ldots+(k-1)c_k}\end{eqnarray*}

\begin{enumerate}

\item The characteristic equation has complex roots. For simplicity, we suppose that $\lambda_{i}=\pm 1,$ for $1\leq i\leq k$,  $\lambda_{j}=\alpha_{j}+i\beta_{j},$ $\overline{\lambda}_{j}=\alpha_{j}-i\beta_{j}$, for $j=k,\ldots,k+s-1$ are $2s$ complex roots with module $1$, and other eigenvalues are less than $1.$ In polar coordinates, we write
$\lambda_{j}=\cos\theta_j+i\sin \theta_j,\overline{\lambda}_{j}=\cos\theta_j-i\sin \theta_j.$ Thus,
\begin{eqnarray*}\gamma_{E}(\Sigma_n)&=&\dsum_{i=1}^{k}(\pm1)^n(a_{i,0}+a_{i,1}n+\cdots+a_{i,m_{i}-1}n^{m_i-1})\\
&&+\dsum_{j=k}^{k+s-1}\cos n\theta_j(a_{j,0}+a_{j,1}n+\cdots+a_{j,m_{j}-1}n^{m_{j}-1})\\
&&+i\dsum_{j=k}^{k+s-1}\sin n\theta_j(a_{j,0}+a_{j,1}n+\cdots+a_{j,m_{j}-1}n^{m_{j}-1})\\
&&+Ad+\frac{d(n+1)}{1+c_2+2c_3+\ldots+(k-1)c_k}+o(1).
\end{eqnarray*}
Note that the expected genus of a graph is a rational number, this implies that $\dsum_{j=k}^{k+s-1}\sin n\theta_j(a_{j,0}+a_{j,1}n+\cdots+a_{j,m_{j}-1}n^{m_{j}-1})$ equals 0. Combining similar terms for the equation, we can get
\begin{eqnarray*}\gamma_{E}(\Sigma_n)&=&a_{1,0}^{'}+a_{1,1}^{'}n+\cdots+a_{1,m-1}^{'}n^{m-1}\\
&&+\dsum_{j=k}^{k+s-1}\cos n\theta_j(a_{j,0}+a_{j,1}n+\cdots+a_{j,m_{j}-1}n^{m_{j}-1})\\
&&+Ad+\frac{d(n+1)}{1+c_2+2c_3+\ldots+(k-1)c_k}+o(1).
\end{eqnarray*}

By lemma \ref{bigO}, the coefficient $a$ in $an^{k}$ of the equation above equals $0$ if $k \geq 2.$

\begin{eqnarray*}\gamma_{E}(\Sigma_n)&=&a_{1,0}^{'}+a_{1,1}^{'}n+(a_{j,0}+a_{j,1}n)\dsum_{j=k}^{k+s-1}\cos n\theta_j\\
&&+Ad+\frac{d(n+1)}{1+c_2+2c_3+\ldots+(k-1)c_k}+o(1)\\
&=&Cn+(a_{j,0}+a_{j,1}n)\dsum_{j=k}^{k+s-1}\cos n\theta_j+B+o(1).
\end{eqnarray*}
 Since each $\cos (\theta_j x)$ is a period function for $j=k,k+1,\ldots,k+s-1,$ their sum $\dsum_{j=k}^{k+s-1}\cos (\theta_j x)$ is periodic. Let $B_n=(a_{j,0}+a_{j,1}n)\dsum_{j=k}^{k+s-1}\cos n\theta_j+B,$ then $\{B_n\}_{n=1}^{\infty}$  is the desired periodic sequence.

 \item All the roots $\lambda_1,\lambda_2,\ldots,\lambda_r$ are real. We suppose that $\lambda_{i_j}=\pm1,$ for $1\leq j\leq s\leq r,$ and all other roots with modulus less $1.$ We have
\begin{eqnarray*}\gamma_{E}(\Sigma_n)&=&\dsum_{j=1}^{s}\lambda_{i_j}^n(a_{i_j,0}+a_{i_j,1}n+\cdots+a_{i_j,m_{i_j}-1}n^{m_{i_j}-1})
\\&&+Ad+\frac{d(n+1)}{1+c_2+2c_3+\ldots+(k-1)c_k}+o(1)\\&=&\dsum_{j=1}^{s}(\pm1)^n(a_{i_j,0}+a_{i_j,1}n+\cdots+a_{i_j,m_{i_j}-1}n^{m_{i_j}-1})
\\&&+Ad+\frac{d(n+1)}{1+c_2+2c_3+\ldots+(k-1)c_k}+o(1).
\end{eqnarray*}
Combining similar terms for the equation above, we can suppose
\begin{eqnarray*}\gamma_{E}(\Sigma_n)&=& a_{1,0}^{'}+a_{1,1}^{'}n+\cdots+a_{1,m-1}^{'}n^{m-1}
\\&&+Ad+\frac{d(n+1)}{1+c_2+2c_3+\ldots+(k-1)c_k}+o(1).
\end{eqnarray*}
From lemma \ref{bigO}, we  have $m\leq 2$ or $m\geq 2$ and $a_{2,1}^{'}=\cdots=a_{2,m-1}^{'}=0.$ Thus, $\gamma_{E}(\Sigma_n)=  Cn+B+o(1).$

\item  $\frac{1}{D}M(1)$ is primitive. By Proposition  \ref{eige}, we assume that $\lambda_1=1,$ and $|\lambda_i|<1,$ for $i=2,3,\ldots,r$
We have
\begin{eqnarray*}\gamma_{E}(\Sigma_n)&=&a_{1,0}+a_{1,1}n+\cdots+a_{1,m_{1}-1}n^{m_1-1}\\&&+Ad+\frac{d(n+1)}{1+c_2+2c_3+\ldots+(k-1)c_k}+o(1).
\end{eqnarray*}
Again by lemma \ref{bigO}, we must have $m_1\leq 2$ or $m_1\geq 2$ and $a_{1,2}=\cdots=a_{1,m_{1}-1}=0.$
Thus, $\gamma_{E}(\Sigma_n)=  Cn+B+o(1).$

\end{enumerate}

\end{proof}

\section{acknowledgement}
\thanks{The author would like to express his deep gratitude to Professor Saul Stahl for the paper he sent me when he was studying for his PhD at Beijing Jiaotong University.}
   %%%%%%%%%%%%%%%%%%%%%%%%%%%%%%%%%%%%%%%%%%%%%%%%%%%%%%%%%%%%%%%%%%%%%%

%%%%%%%%%%%%%%%%%%%%%%%%%%%%%%%%%%%%%%%%%%%%%
%%%%%%%%%%%%%%%%%%%%%%%%%%%%%%%%%%%%%%%%%%%%%

\vskip.51cm
\noindent Version: \printtime\quad\today\quad
\vskip.51cm
\end{document}